\documentclass[journal,twoside,web]{ieeecolor}
\usepackage[utf8]{inputenc}
\usepackage{generic}

\usepackage{bbm}
\usepackage{mathrsfs}
\usepackage{verbatim}
\usepackage{amsmath}
\usepackage{amsfonts}
\usepackage{bm}
\usepackage{graphicx}
\usepackage{caption}
\usepackage{float}

\usepackage{amsthm}
\usepackage{xcolor}
\usepackage[colorinlistoftodos]{todonotes}
\usepackage{mathtools}
\usepackage{cite}
\usepackage{balance}
\usepackage[makeroom]{cancel}

\theoremstyle{definition}

\newtheorem{definition}{Definition}
\newtheorem{remark}{Remark}
\newtheorem{problem}{Problem}

\theoremstyle{plain}

\newtheorem{theorem}{Theorem}
\newtheorem{lemma}{Lemma}
\newtheorem{corollary}{Corollary}

\DeclareMathOperator*{\arginf}{arg\, inf}

\begin{document}

\title{On  Team Decision Problems with Nonclassical Information Structures}

\author{Andreas A. Malikopoulos, {\itshape{Senior Member, IEEE}} 
	\thanks{This research was supported by the Sociotechnical Systems Center (SSC) at the University of Delaware.}%
	\thanks{The author is with the Department of Mechanical Engineering, University of Delaware, Newark, DE 19716 USA (email: \texttt{andreas@udel.edu).}} }


\maketitle



\begin{abstract}
In this paper, we consider  sequential dynamic team decision problems with nonclassical information structures. First, we address the problem from the point of view of a ``manager" who seeks to derive the optimal strategy of the team in a centralized process. We derive structural results that yield an information state for the team which does not depend on the control strategy, and thus it can lead to a dynamic programming decomposition where the optimization problem is over the space of the team's decisions. We, then, derive structural results for each team member that yield an information state which does not depend on their control strategy, and thus it can lead to a dynamic programming decomposition where the optimization problem for each team member is over the space of their decisions. Finally, we show that the solution of each team member is the same as the one derived by the manager. We present an illustrative example of a dynamic team with a delayed sharing information structure.

\end{abstract}

\indent

\begin{IEEEkeywords}
Team theory, decentralized control, non-classical information structures, Markov decision theory.
\end{IEEEkeywords}

\IEEEpeerreviewmaketitle



\section{Introduction} \label{sec:1}

\subsection{Motivation} \label{sec:1a}

\IEEEPARstart{T}{eam} theory \cite{Marshak:1974aa,Radner1962, Marschak_Radner1972} is a mathematical formalism for decentralized stochastic control problems \cite{Sandell:1978aa} in which a ``team," consisting of a number of members, cooperates to achieve a common objective. It was developed to provide a rigorous mathematical framework of cooperating members in which all  members have the same objective yet different information.  
The underlying structure to model a team decision problem consists of \cite{Marschak_Radner1972} (1) a number of $K \in \mathbb{N}$ members of the team; (2) the decisions of each member; (3) the information available to each member, which is different; (4) an objective, which is the same for all members; and (5) the existence, or not, of communication between team members. 
Team theory can be applied effectively in applications that include informationally decentralized systems such as emerging mobility systems  \cite{zhao2019enhanced}, and in particular, optimal coordination of connected and automated vehicles at traffic scenarios \cite{Malikopoulos2017,Cassandras2017, mahbub2020decentralized,Ntousakis:2016aa, Malikopoulos2020, chalaki2020TCST}, networked control systems \cite{Hespanha:2007aa,Zhang:2020aa}, mobility markets \cite{chremos2020MobilityMarket}, smart power grids \cite{Khaitan:2013aa,Howlader:2014aa}, power systems \cite{Du:2017aa}, cooperative cyber-physical networks \cite{Pasqualetti:2015aa,Sami:2016aa,Clark:2017aa}, social media platforms \cite{Dave2020SocialMedia}, cooperation of robots \cite{Saulnier:2017aa,Beaver2020AnFlockingb}, and internet of things \cite{Li:2016aa,Xu:2018aa,Ansere:2020aa}.

\subsection{Related Work} \label{sec:1b}

Team theory  was established with the seminal work of Marschak \cite{Marshak:1974aa}, Radner \cite{Radner1962}, and Marschak and Radner \cite{Marschak_Radner1972} on \textit{static team} problems,  and with Witsenhausen \cite{Witsenhausen1971,Witsenhausen1973} on \textit{dynamic team} problems. In static team problems \cite{Krainak:1982aa,Krainak:1982ab}, the information received by the team members is not affected by the decisions of other team members \cite{Yuksel2013}, while, in dynamic team problems, the information of at least one team member is affected by the decisions of other members  in the team \cite{Yuksel2013}. If there is a prescribed order in which team members make decisions, then such a problem is called a \textit{sequential} team problem. If, however, the team members make decisions in an order that depends on the realization of the team's uncertainty and decisions of other members, then such a problem is called a \textit{non-sequential} team problem. Formulating a well-posed non-sequential team problem is more challenging as we need to ensure that the problem is causal and deadlock free \cite{Andersland:1991aa, Andersland1992,Andersland1994}. 
Teneketzis \cite{Teneketzis1996} presented several results and open questions for non-sequential teams by using the framework of Witsenhausen's intrinsic model \cite{Witsenhausen:1975aa}. In this paper, we restrict our attention to sequential dynamic team decision problems.

The \textit{information structure}  in a sequential team decision problem designates who knows what about the status of the team and when \cite{Schuppen2015,Yuksel2021}. The information structure may designate the complexity \cite{witsenhausen1968counterexample, papadimitriou1982complexity,tsitsiklis1985complexity} of the problem, and can lead to computational implications 
\cite{Papadimitriou:1985aa}. Witsenhausen \cite{Witsenhausen1971a} discussed several information structures and asserted some optimality results for team decision problems.  Ho \cite{Ho:1980aa} investigated information structures within the context of team decision theory using a simple thematic example of a team consisting of two individuals who need to coordinate a meeting. More recently, Mahajan et al. \cite{Mahajan2012} provided a tutorial paper with a comprehensive characterization of information structures. 

Information structures are classified \cite{Schuppen2015} as (1) \textit{classical}, (2) \textit{partially nested} (alternatively also called \textit{overlapping}, or \textit{quasiclassical}), and (3) \textit{nonclassical}. In classical information structures, all team members receive the same information  and have perfect recall \cite{Malikopoulos2016c,Malikopoulos2015,Malikopoulos,Malikopoulos2015b}. If there is only one team member, then such information structures are called \textit{strictly classical} resulting in team decision problems that are typical centralized stochastic control problems \cite{Kumar1986, Kushner1971}. In partially nested information structures, there are some team members who have a nonempty intersection of their information structures while they have perfect recall. Any information structure that is not classical, or partially nested, is called nonclassical and can be further classified \cite{Mahajan2012} as (1) $n$\textit{-step delayed-sharing}, $n\in\mathbb{R}$, where each team member has access to the information, i.e., observations and decisions, of other members after an $n$-step symmetric delay, i.e., same  for all members \cite{Nayyar2011}, or asymmetric delay \cite{Dave2020}; (2) \textit{periodic sharing} \cite{Ooi:1997aa}, where each team member has access to the information, i.e., observations and decisions,  of the other members periodically; (3) \textit{delayed observation (or state} \cite{aicardi1987decentralized}\textit{) sharing information}, where each team member has access to the observations (or states if completely observable) of other members after an $n$-step symmetric, or asymmetric, delay; (4) \textit{delayed control sharing information}, where each team member has access to the decisions of other members after an $n$-step symmetric, or asymmetric, delay \cite{Bismut:1973aa}; and (5) \textit{no sharing information}, where the team members do not share any information. 

Sequential dynamic team problems with nonclassical information structures impose the following technical challenges \cite{Papadimitriou:1987:CMD:2875343.2875347}:
(1) the functional optimization problem of selecting the optimal strategy is not trivial as the class of strategies is infinitely large, and (2) the data increase with time causing significant implications on storage requirements and real-time implementation. In centralized stochastic control theory, these difficulties are addressed  by finding sufficient statistics to compress the growing data without loss of optimality \cite{Striebel1965} using a conditional probability of the state of the team at time $t$ given all the data available up until time $t$. This conditional probability is called \textit{information state}, and it takes values in a time-invariant space. Using this information state can help us derive results for optimal control strategies in a time-invariant domain. Results based on data which, even though they increase with time, are compressed to a sufficient statistic taking values in a time-invariant space are called \textit{structural results} (see \cite{Krishnamurthy2016}, p. 203). 

In centralized stochastic control, structural results can help us establish an information state, which does not depend on the control strategy, and thus they are related to the concept of separation between estimation and control. An important consequence of this separation is that for any given choice of control strategies and a realization of the team's variables until time $t$, the information states at future times do not depend on the choice of the control strategy at time $t$ but only on the realization of the decision at time $t$ (see \cite{Kumar1986}, p. 84). Thus, the future information states are separated from the choice of the current control strategy. The latter is necessary in order to formulate a classical dynamic program \cite{Bertsekas2017,Hordjik1974,Howard}, where at each step the optimization problem is to find the optimal decision for a given realization of the information state \cite{Kumar1986}. 

Several structural results have been reported in the literature to date for team decision problems with nonclassical information structures  \cite{Witsenhausen1971a,Varaiya:1978aa, Kurtaran:1979aa, Nayyar2011,wu2014theory,Gupta:2015aa, Dave2020, Dave2019a,Dave2020a}. However, these results can lead to a sequential decomposition of the optimization problem over a space of functions \cite{Nayyar2011, Nayyar:2019aa,Nayyar2013b} instead of a space of decisions to derive optimal strategies. This is due to the absence of separation between estimation and control which prevents the formulation of a classical dynamic program. There are three general approaches currently in the literature that, in conjunction with these structural results, can be used to derive optimal strategies in sequential dynamic team problems with nonclassical information structures: (1) the \textit{person-by-person} approach, (2) the \textit{designer's} approach, and (3) the \textit{common information} approach. 

The person-by-person approach aims to convert the problem into a centralized stochastic control problem from the point of view of each team member. Namely, we arbitrarily fix the strategies for all team members except for one, say team member $k \in \mathcal{K}$, $\mathcal{K}=\{1,\ldots,K\}$, $K\in\mathbb{N}$. Then, we derive the optimal strategy  for $k$ given the strategies for all other members. We repeat this process for all team members until no  member can improve the performance of the team by unilaterally changing their strategy. Thus, the resulting strategies are person-by-person optimal \cite{Ho:1980aa}. Other research efforts have taken a different approach using Girsanov's change of probability measure to transform the dynamic team problem to a static problem, in which the information structure is not affected by the members' decisions, and then applied the stochastic maximum principle to derive necessary and sufficient optimality conditions for both team and person-by-person strategies \cite{Charalambous:2014aa,Charalambous:2016aa, Charalambous:2015aa, Charalambous:2017aa,Charalambous:2018aa}. An optimal strategy of the team is necessarily person-by-person optimal (see \cite{McGuire1972}, p. 195). However, the converse is not in general true. In addition, if the team payoff function is concave polyhedral, i.e., piecewise linear and concave, then the person-by-person approach is not generally sufficient to determine an optimal strategy (see \cite{McGuire1972}, p. 191), although the problem can be reduced to a linear programming problem. 
Ho \cite{Ho1972a} showed that for a Gaussian team if the observation functions are linear and the cost function is quadratic, then affine control strategies are optimal. The person-by-person approach has  been used in teams with broadcast information structures \cite{Wu2010}, in real-time communication using encoders and decoders \cite{Witsenhausen1979,Varaiya1983,Teneketzis2006,Nayyar2008,Kaspi2010, Yuksel:2013aa}, in quickest detection problems \cite{Teneketzis1984,Veeravalli1993,Tsitsiklis1993,Veeravalli:2001aa}, and networked control systems \cite{Varaiya1983b,Mahajan:2009aa}. 

The designer's approach was first introduced by Witsenhausen \cite{Witsenhausen1973}, as a standard form for sequential stochastic control with a nonclassical information structure, and extended later by Mahajan \cite{mahajan2008sequential} and Yüksel \cite{Yuksel:2020vl}. This approach addresses the team decision problem from the point of view of a ``designer" who knows the team's dynamics and statistics of all sources of uncertainties. Although sequential dynamic team decision problems are informationally decentralized, the designer's approach transforms the problem into a centralized, open-loop planning problem from the designer's point of view where the objective is to derive the strategy of the team before the team starts evolving. Therefore, no data are observed by the designer, and thus this approach leads to a dynamic programming decomposition over a space of functions instead of decisions  imposing significant computational implications  \cite{Papadimitriou:1987:CMD:2875343.2875347}. The person-by-person approach has  been used in conjunction with the designer's approach in real-time communication \cite{Varaiya1983,Mahajan2008,Mahajan:2009ab}, and networked control systems \cite{Varaiya1983b,Mahajan:2009aa}.

The common information approach \cite{Mahajan:2008uq,Nayyar2013b} was first presented for problems with partial history sharing  \cite{Nayyar2011}, where the team members share a subset of their past observations and decisions to a shared memory accessible by all members of the team.  The solution is derived by reformulating the problem from the viewpoint of a ``coordinator" with access only to the shared information (the common information), whose task is to provide prescription strategies to each team member. At each time $t$, the prescription strategies of the team members map their private history of observations and decisions to their optimal decisions at $t$.
The common information approach has been used in problems with control-sharing information structure \cite{Mahajan2013}, in stochastic games with asymmetric information \cite{Nayyar2014}, and in teams with mean-field sharing \cite{Mahajan2015}. There are also some earlier papers that used similar ideas to analyze specific information structures, or structure of the team decision problem, e.g., teams with sequential decompositions \cite{yoshikawa1978decomposition}, teams with partially nested information and common past \cite{casalino1984partially}, teams with delayed state sharing \cite{aicardi1987decentralized}, teams with periodic sharing information structure \cite{Ooi:1997aa}, and teams with belief sharing information structure \cite{yuksel2009stochastic}.

\subsection{Contributions of This Paper} \label{sec:1c}

In this paper, we provide structural results and a classical dynamic programming decomposition of sequential dynamic team decision problems. We first address the problem from the point of view of a ``manager" who seeks to derive the optimal strategy of a team in a centralized process. Then, we address the problem from the point of view of each team member, and show that the solution of each team member is the same as the one derived by the manager. 

The contributions of this paper are the induction of: (1) structural results for the team from the point of view of a manager, i.e., through a centralized process, that yield an information state which does not depend on the control strategy of the team (Theorem \ref{theo:y_t}), and thus it  leads to a classical dynamic programming decomposition where the optimization problem is over the space of the team's decisions (Theorem \ref{theo:dp} and Theorem \ref{theo:dp2}); and (2) structural results for each team member that yield an information state which does not depend on their control strategy (Theorem \ref{theo:y_tk}), and thus it leads to a classical dynamic programming decomposition where the optimization problem is over the space of the decisions of each team member. In addition,  we show that the solution of each team member is the same as the one derived by the manager (Theorem \ref{theo:dp_team}), and therefore, the team members do not need a centralized intervention.

\subsection{Comparison with Related Work} \label{sec:1d}

The one feature which sharply distinguishes previous approaches, reported in Section \ref{sec:1b}, from that undertaken here is that, in this paper, we derive structural results aimed at establishing an information state that does not depend on the control strategy, and thus we can institute separated control strategies that can lead to a classical dynamic programming decomposition. More specifically, the results in this paper advance the state of the art in the following ways. 

First, in contrast to the person-by-person optimal strategy \cite{Witsenhausen1979,Varaiya1983,Varaiya1983b,Teneketzis1984,Veeravalli1993,Tsitsiklis1993,Veeravalli:2001aa, Teneketzis2006,Nayyar2008,Mahajan:2009aa, Kaspi2010,Wu2010, Yuksel:2013aa}, which is not always an optimal strategy of the team (see \cite{McGuire1972}, p. 195), our structural results for each team member (Theorem \ref{theo:y_tk}) guarantee that their optimal control strategies are  also optimal for the team (Theorem \ref{theo:dp_team}). 

Second, while our structural results from the point of view of a manager impose a centralized process, they yield an information state which does not depend on the control strategy of the team (Theorem \ref{theo:y_t}), and thus it can  lead to a classical dynamic programming decomposition where the optimization problem is over the space of the team's decisions (Theorem \ref{theo:dp} and Theorem \ref{theo:dp2}). The designer's approach \cite{Witsenhausen1973, mahajan2008sequential}, on the other hand, transforms the problem into a centralized, open-loop planning problem where the objective is to derive the strategy of the team before the team starts evolving. Therefore, no data are observed by the designer, and thus this approach leads to a dynamic programming decomposition over a space of functions instead of decisions which has significant computational implications  \cite{Papadimitriou:1987:CMD:2875343.2875347}. 

Finally, in contrast to the common information approach \cite{Nayyar2011,Nayyar2013b}, where the coordinator's problem is a centralized stochastic control problem \cite{Nayyar:2014aa} that leads to a dynamic programming decomposition where the optimization problem is over a space of functions, i.e., the prescription functions of the team members, our structural results from the manager's point of view lead to a dynamic programming decomposition where the optimization problem is over the space of the team's decisions (Theorem \ref{theo:dp} and Theorem \ref{theo:dp2}). In addition, our structural results for each team member yield an information state that  leads to a dynamic programming decomposition for each team member resulting  in a solution of each team member which is the same as the one derived by the manager (Theorem \ref{theo:dp_team}), and thus, the team members do not need a centralized intervention.

\subsection{Organization of This Paper} \label{sec:1e}

The remainder of the paper proceeds as follows. In Section II, we provide the modeling framework, information structure, and the optimization problem of a team. In Section III, we derive structural results for the team from the point of view of a manager, and a dynamic programming decomposition where the optimization problem is over the space of the team's decisions. In Section IV, we derive structural results for each team member, and a dynamic programming decomposition where the optimization problem is over the space of the decisions of each team member. In Section V, we present an example of a dynamic team with a delayed sharing information structure consisting of two members. This example was used by Varaiya and Walrand \cite{Varaiya:1978aa} to show that Witsenhausen's structural result asserted in his seminal paper \cite{Witsenhausen1971a} is suboptimal. Finally, we provide concluding remarks and discuss potential directions for future research in Section VI.

\section{Problem Formulation} 
\label{sec:2}

\subsection{Notation}
Subscripts denote time, and superscripts index subsystems. We denote random variables with upper case letters, and their realizations with lower case letters, e.g., for a random variable $X_t$, $x_t$ denotes its realization. The shorthand notation $X_{t}^{1:K}$ denotes the vector of random variables $\big(X_{t}^1, X_{t}^2,\ldots,X_{t}^K\big)$, $x_{t}^{1:K}$ denotes the vector of their realization $\big(x_{t}^1, x_{t}^2,\ldots,x_{t}^K\big)$, and $h^{1:K}_t(\cdot,\cdot)$ denotes the vector of functions $\big(h^1_t(\cdot,\cdot),\ldots, h^K_t(\cdot,\cdot)\big)$. The expectation of a random variable is denoted by $\mathbb{E}[\cdot]$, the probability of an event is denoted by $\mathbb{P}(\cdot)$, and the probability density function is denoted by $p(\cdot)$. 
For a control strategy $\bf{g}$, we use $\mathbb{E}^{\bf{g}}[\cdot]$, $\mathbb{P}^{\bf{g}}(\cdot)$, and $p^{\bf{g}}(\cdot)$ to denote that the expectation, probability, and probability density  function, respectively, depend on the choice of the control strategy $\bf{g}$. For two measurable spaces $(\mathcal{X}, \mathscr{X})$ and $(\mathcal{Y}, \mathscr{Y})$, $\mathscr{X}\otimes\mathscr{Y}$ is the product $\sigma$-algebra on $\mathcal{X}\times \mathcal{Y}$ generated  by the collection of all measurable rectangles, i.e., $\mathscr{X}\otimes\mathscr{Y}\colon= \sigma(\{A\times B: A\in\mathscr{X}, B\in\mathscr{Y} \})$. The product of  $(\mathcal{X}, \mathscr{X})$ and $(\mathcal{Y}, \mathscr{Y})$ is the measurable space $(\mathcal{X}\times \mathcal{Y}, \mathscr{X}\otimes\mathscr{Y})$.

\subsection{Modeling Framework}

We  consider a team of $K \in \mathbb{N}$ members with a measurable state space $(\mathcal{X}_t, \mathscr{X}_t)$, where $\mathcal{X}_t$ is the set in which the team's state takes values at time $t = 0,1,\ldots,T-1$, $T\in\mathbb{N}$, and $\mathscr{X}_t$ is the associated $\sigma$-algebra. The state of the team is  represented by a random variable $X_t: (\Omega, \mathscr{F})\to(\mathcal{X}_t, \mathscr{X}_t),$ defined on the probability space $(\Omega, \mathscr{F}, \mathbb{P})$, where $\Omega$ is the sample space, $\mathscr{F}$ is the associated $\sigma$-algebra, and $\mathbb{P}$ is a probability measure  on $(\Omega, \mathscr{F})$. The decision of each team member $k \in \mathcal{K}$, $\mathcal{K}=\{1,\ldots,K\}$, is represented by a random variable $U_t^k: (\Omega, \mathscr{F})\to(\mathcal{U}_t^k, \mathscr{U}_t^k),$ defined on the probability space $(\Omega, \mathscr{F}, \mathbb{P})$, and takes values in the measurable space $(\mathcal{U}^k_t, \mathscr{U}^k_t)$,  where $\mathcal{U}^k_t$ is  team member $k$'s nonempty feasible set of actions at time $t$ and $\mathscr{U}^k_t$ is the associated $\sigma$-algebra.
Let ${U}_t^{1:K}=(U_t^1,\ldots,U_t^K)$ be the team's decision at time $t$.  Starting at the initial state $X_0$, the evolution of the team is described by the state equation
\begin{align}\label{eq:state}
X_{t+1}=f_t\left(X_t,U_t^{1:K},W_t\right),
\end{align}
where $W_t$ is a random variable defined on the probability space $(\Omega, \mathscr{F}, \mathbb{P})$ that corresponds to the external, uncontrollable disturbance to the team and takes values in a measurable set $(\mathcal{W}, \mathscr{W})$, i.e., $W_t:(\Omega, \mathscr{F})\to(\mathcal{W}, \mathscr{W})$. 
$\{W_t: t=0,\ldots,T-1\}$ is a sequence of independent random variables that are also independent of the initial state $X_0$. 
At time $t = 0,1,\ldots,T-1$, every team member $k\in\mathcal{K}$ makes an observation $Y_t^k$, which takes values in a measurable set $(\mathcal{Y}^k, \mathscr{Y}^k)$, described by the observation equation 
\begin{align}\label{eq:observe}
Y_t^k=h_t^k(X_t,Z_t^k),
\end{align}
where $Z_t^k$ is a random variable defined on the probability space $(\Omega, \mathscr{F}, \mathbb{P})$ that corresponds to the noise of each member's sensor and takes values in a measurable set $(\mathcal{Z}^k, \mathscr{Z}^k)$, i.e., $Z_t^k:(\Omega, \mathscr{F})\to(\mathcal{Z}^k, \mathscr{Z}^k)$.  $\{Z_t^k: t=0,\ldots,T-1; k=1,\ldots,K\}$ is a sequence of independent random variables that are also independent of the initial state $X_0$ and $\{W_t: t=0,\ldots,T-1\}$. 

\subsection{Nonclassical Information Structures}

The team has a nonclassical information structure that can be:
\subsubsection{$n$-step delayed information sharing}

In this case, at time $t$, team member $k\in\mathcal{K}$ observes  $Y_t^k$, and the $n$-step, $n\in\mathbb{R}$, past observations $Y_{0:t-n}^{1:K}$ and decisions $U_{0:t-n}^{1:K}$ of the entire team. Thus, at time $t$, the data available to member $k$ consist of the data $\Delta_t$ available to all team members, i.e.,
\begin{align}\label{eq:delta}
\Delta_t\colon= (Y_{0:t-n}^{1:K}, U_{0:t-n}^{1:K}),
\end{align}
where $Y_{0:t-n}^{1:K}=\{Y_{0:t-n}^{1},\ldots,Y_{0:t-n}^{K}\}$, $U_{0:t-n}^{1:K}=\{U_{0:t-n}^{1},$ $\ldots,U_{0:t-n}^{K}\}$, and  the data $\Lambda_t^k$ known only to member $k\in\mathcal{K},$ i.e.,
\begin{align}\label{eq:lambda}
\Lambda_t^k\colon= (Y_{t-n+1:t}^{k}, U_{t-n+1:t-1}^{k}).
\end{align}
The $n$-step delayed information sharing can  also be asymmetric \cite{Dave2020}, i.e., for each member $k\in\mathcal{K}$, $Y_{t-n_k}^{k}$, $U_{t-n_k}^{k},$ where $n_k\in\mathbb{R}$, is constant but not necessarily the same for each  $k$.

\subsubsection{Periodic information sharing with period $\omega\ge1$}
In this case, for $\alpha= 1,2, \ldots$ and $\alpha\omega< t \le (\alpha +1)\omega$, the pair of
$\Delta_t$ and $\Lambda_t^k$, $k\in\mathcal{K},$ becomes
\begin{align}\label{eq:info1}
	&\Delta_t\colon= (Y_{0:\alpha\omega}^{1:K}, U_{0:\alpha\omega}^{1:K}),\\
	&\Lambda_t^k\colon= (Y_{\alpha\omega+1:(\alpha +1)\omega}^{k}, U_{\alpha\omega+1:(\alpha +1)\omega}^{k}).
\end{align}

\subsubsection{$n$-step delayed observation sharing}
In this case, $\Delta_t$ and $\Lambda_t^k$, $k\in\mathcal{K},$ become
\begin{align}\label{eq:info2}
	&\Delta_t\colon= (Y_{0:t-n}^{1:K}),\\
	&\Lambda_t^k\colon= (Y_{t-n+1:t}^{k}, U_{0:t-1}^{k}).
\end{align}

\subsubsection{$n$-step delayed control sharing}
In this case, $\Delta_t$ and $\Lambda_t^k$, $k\in\mathcal{K},$ become
\begin{align}\label{eq:info3}
	&\Delta_t\colon= (U_{0:t-n}^{1:K}),\\
	&\Lambda_t^k\colon= (Y_{0:t}^{k}, U_{t-n+1:t-1}^{k}).
\end{align}

\subsubsection{No sharing information}
In this case, $\Delta_t$ and $\Lambda_t^k$, $k\in\mathcal{K},$ become
\begin{align}\label{eq:info4}
	&\Delta_t\colon= \emptyset,\\
	&\Lambda_t^k\colon= (Y_{0:t}^{k}, U_{0:t-1}^{k}).
\end{align}

The collection $\{(\Delta_t, \Lambda_t^k);$ $k\in\mathcal{K}; t=0,\ldots,T-1\}$, is the information structure of the team and captures who knows what about the status of the team and when. 

In our exposition, we consider  that the team imposes an $n$-step delayed information sharing, which can be deemed as the general case of a nonclassical information structure.  However, in what follows, the results hold for any special case (2)-(5) above and corresponding $n$. 

\subsection{Optimization Problem}
Let $(\mathcal{D}_t, \mathscr{D}_t)$  and $(\mathcal{L}_t^k, \mathscr{L}_t^k), k\in\mathcal{K},$  be the measurable spaces of all possible realizations of $\Delta_t$ and  $\Lambda_t^k,$ respectively, where $\mathscr{D}_t$ and $\mathscr{L}_t^k$ are the associated $\sigma$-algebras. 
Each team member $k$ makes a decision 
\begin{align}\label{eq:control}
U_{t}^{k}=g_t^k(\Lambda_t^k,\Delta_t),
\end{align}
where $g_t^k$ is a control law of $k\in\mathcal{K}$, which is a measurable function $g_t^k: (\mathcal{L}_t^k\times \mathcal{D}_t,\mathscr{L}_t^k\otimes\mathscr{D}_t)\to (\mathcal{U}^k_t, \mathscr{U}^k_t)$.
The control strategy of team member $k\in\mathcal{K}$ is $\textbf{g}^k=\{g_t^k; ~ t=0,\ldots,T-1\}, \textbf{g}^k\in\mathcal{G}^k,$ where $\mathcal{G}^k$ is the feasible set of the control strategies for $k.$ Thus the set of feasible decentralized control strategies is $\mathcal{G}^{Dec}= \times_{k\in\mathcal{K}}\mathcal{G}^k$, i.e., $\textbf{g}=\{\textbf{g}^1,\ldots, \textbf{g}^K\}\in\mathcal{G}^{Dec}$.
If $\textbf{g}\in\mathcal{G}$ is a centralized control strategy then $\mathcal{G}= : (\mathcal{L}_t^1\times\dots\times\mathcal{L}_t^K\times \mathcal{D}_t, \mathscr{L}_t^1\otimes\dots\otimes\mathscr{L}_t^K\otimes\mathscr{D}_t)$.

\begin{problem} \label{problem1}
The problem is to derive the optimal control strategy $\textbf{g}^*$ of the team that minimizes the expected total cost 
\begin{align}\label{eq:cost}
	J(\textbf{g})=\mathbb{E}^{\textbf{g}}\left[\sum_{t=0}^{T-1} c_t(X_t, U_t^{1:K})+c_T(X_T)\right],
\end{align}
where the expectation is with respect to the joint probability distribution of the random variables $X_t$ and  $U_t^{1:K}$ designated by the choice of $\textbf{g}$, $c_t(X_t, U_t^{1:K}):(\mathcal{X}_t\times \prod_{k\in\mathcal{K}} \mathcal{U}_t^k, \mathscr{X}_t \otimes \mathscr{U}_t^1\otimes\ldots\otimes\mathscr{U}_t^K)\to\mathbb{R}$ 
is the team's  measurable cost function, and $c_T(\hat{X}_T):(\mathcal{X}_T, \mathscr{X}_T) \to\mathbb{R}$ is the measurable cost function at $T$. 

The statistics of the primitive random variables $X_0$,  $\{W_t: t=0,\ldots,T-1\}$, $\{Z_t^k: k\in\mathcal{K};~ t=0,\ldots,T-1\}$, the state equations $\{f_t: t=0,\ldots,T-1\}$, the observation equations $\{h^k_t: k\in\mathcal{K};~ t=0,\ldots,T-1\}$, and the cost functions $\{c_t: t=0,\ldots,T\}$ are all known.
\end{problem}

\section{Structural Results for the Team} \label{sec:3}

We start our exposition by addressing Problem \ref{problem1} from the point of view of a manager who seeks to derive the optimal strategy $\textbf{g}\in\mathcal{G}$ of the team.

\subsection{Information State -- Team}

The first step is to identify an appropriate information state for the team that can be used to formulate a classical dynamic programming decomposition for Problem \ref{problem1}.

\begin{definition} \label{def:infoteam}
	An information state, $\Pi_t$, for the team described by the state equation \eqref{eq:state} (a)  is a function of  $(\Delta_t, \Lambda_t^{1:K})$, and (b) $\Pi_{t+1}$ can be determined from $\Pi_t$, $Y_{t+1}^{1:K}$, and $U_{t}^{1:K}$.
\end{definition}

The notation is simpler if we consider densities for all probability distributions. Let $\textbf{g}\in\mathcal{G}$ be a control strategy and $(\Delta_t, \Lambda_t^{1:K})$ be the information structure of the team. 
To proceed, we first need to prove some essential properties of the conditional probability densities related to the observations of the team members and team's state.

\begin{lemma} \label{lem:y_t}
	For any control strategy $\textbf{g}\in\mathcal{G}$ of the team, 
	\begin{align}\label{eq:y_t}
		p^{\textbf{g}}(Y^{1:K}_{t+1}~|~X_{t+1}, \Delta_{t}, \Lambda^{1:K}_t, U^{1:K}_t)= p(Y^{1:K}_{t+1}~|~X_{t+1}),
	\end{align}
for all $t=0,1,\ldots, T-1.$
\end{lemma}
\begin{proof}
	 The realization of  $Y^{1:K}_{t+1}$ is statistically determined by the conditional distribution of $Y^{1:K}_{t+1}$ given $X_{t+1}$ in \eqref{eq:observe}, hence
	\begin{align}\label{eq:lem1a}
			p^{\textbf{g}}(Y^{1:K}_{t+1}~|~X_{t+1}, \Delta_{t}, \Lambda^{1:K}_t, U^{1:K}_t)= p^{\textbf{g}}(Y^{1:K}_{t+1}~|~X_{t+1}).
	\end{align}		
However,
	\begin{align}\label{eq:lem1b}
			p^{\textbf{g}}(Y^{1:K}_{t+1}~|~X_{t+1}) = p^{\textbf{g}}(Z^{1:K}_{t+1}\in \prod_{k\in\mathcal{K}} B^k~|~X_{t+1}),
	\end{align}	
	where $B^k\in \mathscr{Z}^k$, $k\in\mathcal{K}$. Since, $\{Z_{t}^k:~k=1,\ldots,K;~ t=0,\ldots,T-1\}$ is a sequence of independent random variables that are independent of $X_{t+1}$, 
	\begin{align}\label{eq:lem1c}
				p^{\textbf{g}}(Z^{1:K}_{t+1}\in \prod_{k\in\mathcal{K}} B^k~|~X_{t+1}) = p(Z^{1:K}_{t+1}\in \prod_{k\in\mathcal{K}} B^k).
	\end{align}
 Thus, 
	\begin{align}\label{eq:lem1d}
	p^{\textbf{g}}(Y^{1:K}_{t+1}~|~X_{t+1}) = p(Y^{1:K}_{t+1}~|~X_{t+1}).
	\end{align} 
  The result follows from \eqref{eq:lem1a} and \eqref{eq:lem1d}.
\end{proof}

\begin{lemma} \label{lem:x_t1ut}
	For any control strategy $\textbf{g}\in\mathcal{G}$ of the team, 
	\begin{gather}
		p^{\textbf{g}}(X_{t+1}~|~X_t, \Delta_{t}, \Lambda^{1:K}_t, U^{1:K}_t) = p(X_{t+1}~|~X_t, U^{1:K}_t), \label{eq:x_t1ut}
	\end{gather}
	for all $t=0,1,\ldots, T-1.$
\end{lemma}
\begin{proof}
	The realization of  $X_{t+1}$ is statistically determined by the conditional distribution of $X_{t+1}$ given $X_{t}$ and $U^{1:K}_{t}$, i.e., by $p^{\textbf{g}}(X_{t+1}~|~X_t,  U^{1:K}_t) $. From \eqref{eq:state}, we have
	\begin{gather}
		p^{\textbf{g}}(X_{t+1}~|~X_t,  U^{1:K}_t) = p^{\textbf{g}}(W_{t}\in A~|~X_t,  U^{1:K}_t), \label{eq:lem2a}
	\end{gather}		
	where $A\in\mathscr{W}$. Since, $\{W_{t}: t=0,\ldots,T-1\}$ is a sequence of independent random variables that are independent of $X_{t}$ and $U^{1:K}_{t}$, 
	\begin{align}
		p^{\textbf{g}}(W_{t}\in A~|~X_t,  U^{1:K}_t) = p(W_{t}\in A).\label{eq:lem2b}
	\end{align}

Next,
	\begin{align}
		&p^{\textbf{g}}(X_{t+1}~|~X_t, \Delta_{t}, \Lambda^{1:K}_t, U^{1:K}_t) \nonumber\\
		&= p^{\textbf{g}}(W_{t}\in A~|~X_t, \Delta_{t}, \Lambda^{1:K}_t, U^{1:K}_t) = p(W_{t}\in A). \label{eq:lem2c}
	\end{align}	
 The result follows from \eqref{eq:lem2a}, \eqref{eq:lem2b} and \eqref{eq:lem2c}.	
\end{proof}

\begin{lemma} \label{lem:x_t}
	For any control strategy $\textbf{g}\in\mathcal{G}$ of the team, 
	\begin{align}\label{eq:x_t}
		p^{\textbf{g}}(X_{t}~|~\Delta_{t}, \Lambda^{1:K}_t) = p(X_{t}~|~\Delta_{t}, \Lambda^{1:K}_t),
	\end{align}
	for all $t=0,1,\ldots, T-1.$
\end{lemma}
\begin{proof} 
We have
	\begin{align}
		&p^{\textbf{g}}(X_{t}~|~\Delta_{t}, \Lambda^{1:K}_t)\nonumber\\
		&= p^{\textbf{g}}(X_{t}~|~\Delta_{t}, \Lambda^{1:K}_{t-2}, Y^{1:K}_{t-1}, Y^{1:K}_t, U^{1:K}_{t-2},U^{1:K}_{t-1}). \label{eq:lem4a}
	\end{align}		
	However, the realization of $X_{t}$ is statistically determined by the conditional distribution of  $X_{t}$ given $X_{t-1}$ and $U^{1:K}_{t-1}$, which does not depend on the control strategy $\textbf{g}$ (Lemma \ref{lem:x_t1ut}), so we can drop the superscript in \eqref{eq:lem4a}, and thus \eqref{eq:x_t} follows immediately.	
\end{proof}


\begin{remark}\label{cor:lemU}
	As a consequence of Lemma \ref{lem:x_t}, and since $X_{t}$ does not depend on $U^{1:K}_t$, we have 
		\begin{align}\label{eq:lemU}
		p^{\textbf{g}}(X_{t}~|~\Delta_{t}, \Lambda^{1:K}_t, U^{1:K}_t) = p(X_{t}~|~\Delta_{t}, \Lambda^{1:K}_t).
		\end{align}
\end{remark}
Given that the manager can observe the data $(\Delta_t, \Lambda_t^{1:K})$ of the team, our hypothesis is that we can compress these data to a sufficient statistic of the state of the team. 
This statistic is the probability density function $p(X_{t}~|~\Delta_{t}, \Lambda^{1:K}_{t})$.
The next result proves our hypothesis and shows that such information state does not depend on the team's control strategy.

\begin{theorem}[Information State -- Team] \label{theo:y_t}
	For any control strategy $\textbf{g}\in\mathcal{G}$ of the team, the conditional probability density $p^{\textbf{g}}(X_{t}~|~\Delta_{t}, \Lambda^{1:K}_{t})$ does not depend on the control strategy $\textbf{g}$.
	It is an information state $\Pi_{t}(\Delta_{t}, \Lambda^{1:K}_{t})(X_{t})$, i.e., $\Pi_{t}(\Delta_{t}, \Lambda^{1:K}_{t})(X_{t})=p(X_{t}~|~\Delta_{t}, \Lambda^{1:K}_{t})$ with $\int_{\mathscr{X}_t} \Pi_{t}(\Delta_{t},$ $\Lambda^{1:K}_{t})(X_t) dX_t =1$,	
	 that can be evaluated from  $\Delta_{t}, \Lambda^{1:K}_{t}$. Moreover, there is a function $\theta_t$, which does not depend on the control strategy $\textbf{g}$, such that
	\begin{align}\label{eq:xt1}
		\Pi_{t+1}(\Delta_{t+1}, \Lambda^{1:K}_{t+1})(X_{t+1})& \nonumber\\
		= \theta_t\big[ \Pi_{t}(\Delta_{t}, \Lambda^{1:K}_{t})&(X_{t}), Y^{1:K}_{t+1}, U^{1:K}_t \big],
	\end{align}
	for all $t=0,1,\ldots, T-1.$
\end{theorem}
\begin{proof}
		 See Appendix \ref{app:1}.
\end{proof}

Note that the information state $\Pi_{t+1}(\Delta_{t+1}, \Lambda^{1:K}_{t+1})(X_{t+1})=p(X_{t+1}~|~\Delta_{t+1}, \Lambda^{1:K}_{t+1})$ of the team is the entire probability density function and not just its value at any particular realization of  $(\Delta_{t+1}, \Lambda_{t+1}^{1:K})$. This is because to compute $\Pi_{t+1}(\Delta_{t+1}, \Lambda^{1:K}_{t+1})(X_{t+1})$ for any particular $X_{t+1}$, we need the probability density functions $p(~\cdot ~|~ \Delta_{t}, \Lambda^{1:K}_t, U^{1:K}_t)$ and $p(~\cdot ~|~ \Delta_{t}, \Lambda^{1:K}_t)$. This implies that the information state takes values in the space of these probability densities on the measurable space $(\mathcal{X}_t, \mathscr{X}_t)$, which is an infinite-dimensional space.

\subsection{Optimal Control Strategy of the Team}

In what follows, to simplify notation, the information state $\Pi_{t}(\Delta_{t}, \Lambda^{1:K}_{t})$ of the team at $t$ is denoted simply by $\Pi_t$. We use its arguments $\Delta_{t}$ and $\Lambda^{1:K}_{t}$ only if it is required by our exposition.

\begin{definition}\label{def:septeam}
	A control strategy $\textbf{g}=\{g_t;~ t=0,\ldots,T-1\}$ is said to be \textit{separated} if $g_t$ depends on $\Delta_{t}$ and $\Lambda^{1:K}_{t}$ only through the information state, i.e., $U^{1:K}_t  = g_t\big(\Pi_{t}(\Delta_{t}, \Lambda^{1:K}_{t})\big)$. Let $\mathcal{G}^s\subseteq\mathcal{G}$ denote the set of all separated control strategies.
\end{definition}

In implementing a separated control strategy, we first need to compute the conditional probability $\Pi_{t}(\Delta_{t}, \Lambda^{1:K}_{t})$, and then choose the control, since the task of estimation and control are separated.
Next, we use the information state to define recursive functions which are analogous to the comparison principle (see \cite{Kumar1986}, p. 74).

\begin{theorem} \label{theo:dp}
	Let $V_t\big(\Pi_{t}(\Delta_{t}, \Lambda^{1:K}_{t})\big)$ be functions defined recursively for all $\textbf{g}\in\mathcal{G}^s$ by
	\begin{align}
			&V_T\big(\Pi_{T}(\Delta_{T}, \Lambda^{1:K}_{T})\big)\coloneqq \mathbb{E}^{\textbf{g}}\Big[c_T(X_T)~|~\Pi_{T}=\pi_T \Big],\label{theo2:1a}
	\end{align}
	\begin{align}			
			&V_t\big(\Pi_{t}(\Delta_{t}, \Lambda^{1:K}_{t})\big)\coloneqq \inf_{u^{1:K}_t\in\prod_{k\in\mathcal{K}} \mathcal{U}_t^k }\mathbb{E}^{\textbf{g}}\Big[c_t(X_t,U^{1:K}_t)\nonumber\\ 
			&+ V_{t+1}\big(\theta_t\big[ \Pi_{t}(\Delta_{t}, \Lambda^{1:K}_{t}), Y^{1:K}_{t+1}, U^{1:K}_t\big]\big)~|~\Pi_{t}=\pi_t, \nonumber\\
			&U^{1:K}_t=u^{1:K}_t  \Big], \label{theo2:1b}
	\end{align}
	where $c_T(X_T)$ is the cost function at $T$, and $\pi_T$, $\pi_t$, $u^{1:K}_t$ are the realizations of $\Pi_{T}$, $\Pi_{t}$, and $U^{1:K}_t$, respectively.
	Then, for any control strategy $\textbf{g}\in\mathcal{G}$,
	\begin{align}			
		V_t\big(\Pi_{t}(\Delta_{t}, \Lambda^{1:K}_{t})\big)\le J_t(\textbf{g})\coloneqq \mathbb{E}^{\textbf{g}}\Big[\sum_{l=t}^{T-1}c_l(X_l,U^{1:K}_l)&\nonumber\\ 
		+ c_T(X_T) ~|~\Delta_{t}, \Lambda^{1:K}_{t} \Big],& \label{theo2:1c}	
	\end{align}	
	where $J_t(\textbf{g})$ is the cost-to-go of the team at time $t$ corresponding to the control strategy $\textbf{g}\in\mathcal{G}.$
\end{theorem}
	\begin{proof}
We prove \eqref{theo2:1c} by induction. For $t=T$,
	\begin{align}			
		J_T(\textbf{g})\coloneqq& ~\mathbb{E}^{\textbf{g}}\Big[c_T(X_T) |~\Delta_{T}, \Lambda^{1:K}_{T} \Big] \nonumber\\
		=& \int_{\mathscr{X}_T} c_T(X_T) ~\Pi_{T}(\Delta_{T}, \Lambda^{1:K}_{T})(X_T)~ dX_T,
		\label{theo2:1d}	
	\end{align}	
and so \eqref{theo2:1c} holds with equality. 

Suppose that \eqref{theo2:1c} holds for $t+1$. Then,
	\begin{align}			
	&J_t(\textbf{g})= \mathbb{E}^{\textbf{g}}\Big[\sum_{l=t}^{T-1}c_l(X_l,U^{1:K}_l)
	+ c_T(X_T) ~|~\Delta_{t}, \Lambda^{1:K}_{t} \Big]\nonumber\\ 
	&= \mathbb{E}^{\textbf{g}}\Big[c_t(X_t,U^{1:K}_t) + \sum_{l=t+1}^{T-1}c_l(X_l,U^{1:K}_l)\nonumber\\
	&+ c_T(X_T) ~|~\Delta_{t}, \Lambda^{1:K}_{t} \Big]\nonumber\\
	&=  \mathbb{E}^{\textbf{g}}\bigg[ \mathbb{E}^{\textbf{g}}  \Big[c_t(X_t,U^{1:K}_t) \nonumber\\ 
	&+\sum_{l=t+1}^{T-1}c_l(X_l,U^{1:K}_l)+ c_T(X_T)~|~\Delta_{t}, \Lambda^{1:K}_{t}, U^{1:K}_{t} \Big] \nonumber\\
	&|~\Delta_{t}, \Lambda^{1:K}_{t}\bigg] \nonumber\\
	&\ge  \mathbb{E}^{\textbf{g}}\bigg[ \mathbb{E}^{\textbf{g}}  \Big[c_t(X_t,U^{1:K}_t) \nonumber\\ 
	&+V_{t+1}\big( \theta_t\big[ \Pi_{t}(\Delta_{t}, \Lambda^{1:K}_{t}), Y^{1:K}_{t+1}, U^{1:K}_t \big]\big) ~|~\Pi_{t}(\Delta_{t}, \Lambda^{1:K}_{t}), \nonumber\\
	&U^{1:K}_{t} \Big] ~|~\Delta_{t}, \Lambda^{1:K}_{t}\bigg]\nonumber\\
	&=  \mathbb{E}^{\textbf{g}}\bigg[ V_{t}\big(  \Pi_{t}(\Delta_{t}, \Lambda^{1:K}_{t})\big)  ~|~\Delta_{t}, \Lambda^{1:K}_{t}\bigg]=V_{t}\big(  \Pi_{t}(\Delta_{t}, \Lambda^{1:K}_{t})\big), 	
	\label{theo2:1e}	
	\end{align}	
where, in the inequality, we used the hypothesis and, in the last equality, we used \eqref{theo2:1b}. Thus, \eqref{theo2:1c} holds for all $t$.
	\end{proof}

In view of Theorem \ref{theo:dp}, we show that an optimal strategy of the team is separated and obtain a classical dynamic programming decomposition where the optimization problem is over the space of the team's decisions.

\begin{theorem} \label{theo:dp2}
	Let 
	\begin{align}			
		&V_t\big(\Pi_{t}(\Delta_{t}, \Lambda^{1:K}_{t})\big)\coloneqq \inf_{u^{1:K}_t\in\prod_{k\in\mathcal{K}} \mathcal{U}_t^k }\mathbb{E}^{\textbf{g}}\Big[c_t(X_t,U^{1:K}_t)\nonumber\\ 
		&+ V_{t+1}\big(\theta_t\big[ \Pi_{t}(\Delta_{t}, \Lambda^{1:K}_{t}), Y^{1:K}_{t+1}, U^{1:K}_t \big]\big)~|~\Pi_{t}=\pi_t, \nonumber\\
		&U^{1:K}_t=u^{1:K}_t  \Big], \label{theo3:1a}
	\end{align}
	and let $\textbf{g}\in\mathcal{G}^s$ be a separated control strategy that achieves the infimum in \eqref{theo3:1a}. Then $\textbf{g}\in\mathcal{G}^s$ is optimal and 
	\begin{align}			
	V_t\big(\Pi_{t}(\Delta_{t}, \Lambda^{1:K}_{t})\big)=J_t(\textbf{g}), \label{theo3:1b}
\end{align}	
with probability $1$.
\end{theorem}
\begin{proof}
	We first prove \eqref{theo3:1b} by induction. For $t=T$,
	\begin{align}			
		J_T(\textbf{g})\coloneqq& ~\mathbb{E}^{\textbf{g}}\Big[c_T(X_T) ~|~\Delta_{T}, \Lambda^{1:K}_{T} \Big] \nonumber\\
		=& \int_{\mathscr{X}_T} c_T(X_T) \Pi_{T}(\Delta_{T}, \Lambda^{1:K}_{T})(X_T) dX_T.
		\label{theo3:1c}	
	\end{align}	
Suppose that \eqref{theo3:1a} holds for $t+1$. Then
	\begin{align}			
	&\inf_{u^{1:K}_t\in\prod_{k\in\mathcal{K}} \mathcal{U}_t^k }\mathbb{E}^{\textbf{g}}\Big[\sum_{l=t}^{T-1}c_l(X_l,U^{1:K}_l)
	+ c_T(X_T) |~\Delta_{t}, \Lambda^{1:K}_{t} \Big]\nonumber\\ 
	&= \inf_{u^{1:K}_t\in\prod_{k\in\mathcal{K}} \mathcal{U}_t^k }\mathbb{E}^{\textbf{g}}\Big[c_t(X_t,U^{1:K}_t) + \sum_{l=t+1}^{T-1}c_l(X_l,U^{1:K}_l)\nonumber\\
	&+ c_T(X_T) ~|~\Delta_{t}, \Lambda^{1:K}_{t} \Big]\nonumber
	\end{align}
	\begin{align}	
	&=\inf_{u^{1:K}_t\in\prod_{k\in\mathcal{K}} \mathcal{U}_t^k }  \mathbb{E}^{\textbf{g}}\bigg[ \mathbb{E}^{\textbf{g}}  \Big[c_t(X_t,U^{1:K}_t) \nonumber\\ 
	&+\sum_{l=t+1}^{T-1}c_l(X_l,U^{1:K}_l)+ c_T(X_T)|~\Delta_{t}, \Lambda^{1:K}_{t}, U^{1:K}_{t} \Big] \nonumber\\
	&|~\Delta_{t}, \Lambda^{1:K}_{t}\bigg] \nonumber\\
	&= \inf_{u^{1:K}_t\in\prod_{k\in\mathcal{K}} \mathcal{U}_t^k } \mathbb{E}^{\textbf{g}}\bigg[ \mathbb{E}^{\textbf{g}}  \Big[c_t(X_t,U^{1:K}_t) \nonumber\\ 
	&+V_{t+1}\big( \theta_t\big[ \Pi_{t}(\Delta_{t}, \Lambda^{1:K}_{t}), Y^{1:K}_{t+1}, U^{1:K}_t \big]\big) ~|~\Pi_{t}(\Delta_{t}, \Lambda^{1:K}_{t}), \nonumber\\
	&U^{1:K}_{t} \Big] ~|~\Delta_{t}, \Lambda^{1:K}_{t}\bigg]\nonumber\\
	&= \mathbb{E}^{\textbf{g}}\bigg[ V_{t}\big(  \Pi_{t}(\Delta_{t}, \Lambda^{1:K}_{t})\big)  ~|~\Delta_{t}, \Lambda^{1:K}_{t}\bigg] =V_{t}\big(  \Pi_{t}(\Delta_{t}, \Lambda^{1:K}_{t})\big), 	
	\label{theo3:1d}	
	\end{align}	
where, in the third equality, we used the hypothesis and, in the forth equality, $u^{1:K}_t$ achieves the infimum. Thus, \eqref{theo3:1a} holds for all $t$.

For $t=0$, \eqref{theo3:1b} yields $J_0(\textbf{g})=V_0\big(\Pi_{0}(\Delta_{0}, \Lambda^{1:K}_{0})\big)$. Taking expectations
	\begin{align}			
		J(\textbf{g})= \mathbb{E}^{\textbf{g}}\Big[J_0(\textbf{g})\Big]
 = \mathbb{E}^{\textbf{g}}\Big[V_0\big(\Pi_{0}(\Delta_{0}, \Lambda^{1:K}_{0})\big)  \Big].
	\label{theo3:1e}	
	\end{align}	
By Theorem \ref{theo:dp}, it follows that for any other $\textbf{g}'\in\mathcal{G}$,  
	\begin{align}			
	J(\textbf{g}')\ge \mathbb{E}^{\textbf{g}}\Big[V_0\big(\Pi_{0}(\Delta_{0}, \Lambda^{1:K}_{0})\big)  \Big].
	\label{theo3:1f}	
	\end{align}	
\end{proof}

The following results are derived by using ideas from \cite{Sondik1971, Smallwood1973}.

\begin{lemma} \label{lemma:homo}
	Let $V_t\big(\Pi_{t}(\Delta_{t}, \Lambda^{1:K}_{t})\big)$ be functions defined recursively for all $\textbf{g}\in\mathcal{G}^s$ by
\begin{align}
	&V_T\big(\Pi_{T}(\Delta_{T}, \Lambda^{1:K}_{T})\big)\coloneqq \mathbb{E}^{\textbf{g}}\Big[c_T(X_T)~|~\Pi_{T}=\pi_T \Big],\label{homo:1a}
\end{align}
\begin{align}			
	&V_t\big(\Pi_{t}(\Delta_{t}, \Lambda^{1:K}_{t})\big)\coloneqq \inf_{u^{1:K}_t\in\prod_{k\in\mathcal{K}} \mathcal{U}_t^k }\mathbb{E}^{\textbf{g}}\Big[c_t(X_t,U^{1:K}_t)\nonumber\\ 
	&+ V_{t+1}\big(\theta_t\big[ \Pi_{t}(\Delta_{t}, \Lambda^{1:K}_{t}), Y^{1:K}_{t+1}, U^{1:K}_t\big]\big)~|~\Pi_{t}=\pi_t, \nonumber\\
	&U^{1:K}_t=u^{1:K}_t  \Big], \label{homo:1b}
\end{align}
where $c_T(X_T)$ is the cost function at $T$, and $\pi_T$, $\pi_t$, $u^{1:K}_t$ are the realizations of $\Pi_{T}$, $\Pi_{t}$, and $U^{1:K}_t$, respectively.
Then, for all $t=0,\ldots,T$, $V_t\big(\Pi_{t}(\Delta_{t}, \Lambda^{1:K}_{t})\big)$ is positive homogeneous, i.e., for any  $\rho >0$, $V_t\big(\rho~ \Pi_{t}(\Delta_{t}, \Lambda^{1:K}_{t})\big)= \rho ~V_t\big(\Pi_{t}(\Delta_{t}, \Lambda^{1:K}_{t})\big)$. 
\end{lemma}
\begin{proof}
See Appendix \ref{app:2}.
\end{proof}

\begin{theorem} \label{theo:concave}
	Let 
\begin{align}			
	&V_t\big(\Pi_{t}(\Delta_{t}, \Lambda^{1:K}_{t})\big)\coloneqq \inf_{u^{1:K}_t\in\prod_{k\in\mathcal{K}} \mathcal{U}_t^k }\mathbb{E}^{\textbf{g}}\Big[c_t(X_t,U^{1:K}_t)\nonumber\\ 
	&+ V_{t+1}\big(\theta_t\big[ \Pi_{t}(\Delta_{t}, \Lambda^{1:K}_{t}), Y^{1:K}_{t+1}, U^{1:K}_t\big]\big)~|~\Pi_{t}=\pi_t, \nonumber\\
	&U^{1:K}_t=u^{1:K}_t  \Big], \label{theo_homo:1a}
\end{align}
where $c_T(X_T)$ is the cost function at $T$, and $\pi_T$, $\pi_t$, $u^{1:K}_t$ are the realizations of $\Pi_{T}$, $\Pi_{t}$, and $U^{1:K}_t$, respectively.
Then,  for all $t=0,\ldots,T$, $V_t\big(\Pi_{t}(\Delta_{t}, \Lambda^{1:K}_{t})\big)$ is concave with respect to $\Pi_{t}(\Delta_{t}, \Lambda^{1:K}_{t})$.
\end{theorem}
\begin{proof}
See Appendix \ref{app:3}.	
\end{proof}

\section{Structural Results for the Team Members} \label{sec:4}

In this section, we address Problem \ref{problem1} from the point of view of a team member $k\in\mathcal{K}$ who seeks to derive their optimal strategy $\textbf{g}^k=\{g_t^k; ~k\in\mathcal{K};~ t=0,\ldots,T-1\}$ which will constitute the team's strategy $\textbf{g}=\{\textbf{g}^1,\ldots, \textbf{g}^K\}\in\mathcal{G}^{Dec}$.

\subsection{Information State -- Team Members}

We first identify an appropriate information state for member $k\in\mathcal{K}$ that can be used to formulate a classical dynamic programming decomposition for Problem \ref{problem1}.

\begin{definition} \label{def:infomember}
	An information state, $\Pi_t^k$, for member $k\in\mathcal{K}$ of a team described by the state equation \eqref{eq:state}  (a)  is a function of  $(\Delta_t, \Lambda_t^{k})$, and (b) $\Pi_{t+1}^k$ can be determined from $\Pi_t^k$, $Y_{t-n+1}^{1:K}$, $U_{t-n+1}^{1:K}$, $Y_{t+1}^{k}$, and $U_{t}^{k}$.
\end{definition}

To proceed,  we first need to prove some essential properties of the conditional probabilities densities related to the observation of team member $k\in\mathcal{K}$ and the team's state.

\begin{lemma} \label{lem:y_tk}
	For any  control strategy $\textbf{g}=\{\textbf{g}^1,\ldots, \textbf{g}^K\}$ of the team, 
	\begin{align}\label{eq:y_tk}
		p^{\textbf{g}}(Y^k_{t+1}~|~X_{t+1}, \Delta_{t+1}, \Lambda^k_t, U^k_t)= p(Y^k_{t+1}~|~X_{t+1}),
	\end{align}
	for all $t=0,1,\ldots, T-1.$
\end{lemma}
\begin{proof}
	The realization of  $Y^k_{t+1}$ is statistically determined by the conditional distribution of $Y^k_{t+1}$ given $X_{t+1}$ in \eqref{eq:observe}, hence
	\begin{align}\label{eq:lemy_tka}
		p^{\textbf{g}}(Y^k_{t+1}~|~X_{t+1}, \Delta_{t+1}, \Lambda^k_t, U^k_t)= p^{\textbf{g}}(Y^k_{t+1}~|~X_{t+1}).
	\end{align}		
	However,
	\begin{align}\label{eq:lemy_tkb}
		p^{\textbf{g}}(Y^k_{t+1}~|~X_{t+1}) = p^{\textbf{g}}(Z^k_{t+1}\in B^k~|~X_{t+1}),
	\end{align}	
	where $B^k\in\mathscr{Z}^k$. Since, $\{Z_{t}^k: t=0,\ldots,T-1; k=1,\ldots,K\}$ is a sequence of independent random variables that are independent of $X_{t+1}$, we have
	\begin{align}\label{eq:lemy_tkc}
		p^{\textbf{g}}(Z^k_{t+1}\in B^k~|~X_{t+1}) = p(Z^k_{t+1}\in B^k).
	\end{align}
	Thus, 
	\begin{align}\label{eq:lemy_tkd}
		p^{\textbf{g}}(Y^k_{t+1}~|~X_{t+1}) = p(Y^k_{t+1}~|~X_{t+1}).
	\end{align} 
	The result follows from \eqref{eq:lemy_tka} and \eqref{eq:lemy_tkd}.
\end{proof}

\begin{lemma} \label{lem:x_t1k}
	For any  control strategy $\textbf{g}=\{\textbf{g}^1,\ldots, \textbf{g}^K\}$ of the team, $	p^{\textbf{g}}(X_{t+1}~|~\Delta_{t+1}, \Lambda^k_t, U^k_t)$ does not depend on the control strategy $\textbf{g}^k$ of member $k$. It depends only on the strategy $\textbf{g}^{-k}=(\textbf{g}^1,\ldots,\textbf{g}^{k-1}, \textbf{g}^{k+1}, \ldots,\textbf{g}^K),$ of the other team members, i.e.,
	\begin{align}
		p^{\textbf{g}}(X_{t+1}~|~ \Delta_{t+1}, \Lambda^k_t, U^k_t) = p^{\textbf{g}^{-k}}(X_{t+1}~|~ \Delta_{t+1}, \Lambda^k_t, U^k_t), \label{eq:x_t1k}
	\end{align}
	for all $t=0,1,\ldots, T-1.$		
\end{lemma}
\begin{proof} 	
	Since  $X_{t+1}^k=f_{t}(X_{t}, U_{t}^{1:K},W_t)$ and $\Lambda_t^k\colon= (Y_{t-n+1:t}^{k}, U_{t-n+1:t-1}^{k})$, and $Y_{t}^k = h_t^k(X_{t}, Z_{t}^k)$ we have
	\begin{align}
		&p^{\textbf{g}}(X_{t+1} ~|~\Delta_{t+1}, \Lambda^k_t,  U^k_t )\nonumber\\
		&= p^{\textbf{g}}(W_{t}\in A, U^{-k}_{t}, Z^k_{t}\in B^k ~|~\Delta_{t+1}, \Lambda^k_t,  U^k_t )\nonumber\\
		&= p^{\textbf{g}}(W_{t}\in A ~|~\Delta_{t+1}, \Lambda^k_t,  U^k_t ) \cdot p^{\textbf{g}}(U^{-k}_{t}~|~\Delta_{t+1}, \Lambda^k_t,  U^k_t )\nonumber\\
		&\cdot p^{\textbf{g}}( Z^k_{t}\in B^k ~|~\Delta_{t+1}, \Lambda^k_t,  U^k_t )\nonumber\\
		&= p(W_{t}\in A ~|~\Delta_{t+1}, \Lambda^k_t,  U^k_t ) \cdot p^{\textbf{g}{-k}}(U^{-k}_{t} ~|~\Delta_{t+1})\nonumber\\
		&\cdot p( Z^k_{t}\in B^k ~|~\Delta_{t+1}, \Lambda^k_t,  U^k_t ),
		\label{lem:x_t1k_a}
	\end{align}	
where $A\in\mathscr{W}$, $B^k\in\mathscr{Z}^k$, and $U^{-k}_{t}=(U_{t}^1,\ldots,U_{t}^{k-1},$   $U_{t}^{k+1}, \ldots,U_{t}^K)$. In the last equality, the second term depends only on $\textbf{g}^{-k}$ while we dropped the superscript $\textbf{g}$ in $p(W_{t}\in A ~|~\Delta_{t+1}, \Lambda^k_t,  U^k_t )$ and $p( Z^k_{t}\in B^k ~|~\Delta_{t+1}, \Lambda^k_t,  U^k_t )$ since both $W_{t}$ and $Z^k_{t}$ are two sequences of independent random variables, and the evolution of their probability measure does not depend on the control strategy.		
\end{proof}

\begin{theorem} [Information State -- Team Members]\label{theo:y_tk}
	For any control strategy $\textbf{g}=\{\textbf{g}^1,\ldots, \textbf{g}^K\}\in\mathcal{G}^{Dec}$ of the team, the conditional probability density $p(X_{t}~|~\Delta_{t}, \Lambda^{k}_{t})$ does not depend on the control strategy $\textbf{g}^k$ of member $k$. It depends only on the strategy $\textbf{g}^{-k}=(\textbf{g}^1,\ldots,\textbf{g}^{k-1}, \textbf{g}^{k+1}, \ldots,\textbf{g}^K)$ of the other team members. It is an information state of the team member $k$, i.e., $\Pi_{t}^k(\Delta_{t}, \Lambda^{k}_{t})(X_{t})=p^{\textbf{g}^{-k}}(X_{t}~|~\Delta_{t}, \Lambda^{k}_{t}),$ that can be evaluated from  $\Delta_{t}$ and  $\Lambda^{k}_{t}$. Moreover, there is a function $\theta_t^k$, which does not depend on the control strategy $\textbf{g}^k$ of member $k$, such that
	\begin{align}\label{eq:xt1aef}
		&\Pi_{t+1}^k(\Delta_{t+1}, \Lambda^{k}_{t+1})(X_{t+1})\nonumber\\
		&=\theta^k_t\big[ \Pi^k_{t}(\Delta_{t}, \Lambda^k_{t})(X_{t}),  \Delta_{t+1}, \Lambda^{k}_{t+1}  \big],
	\end{align}
	for all $t=0,1,\ldots, T-1.$
\end{theorem}
\begin{proof}
	See Appendix \ref{app:4}.
\end{proof}

Note that the information state $\Pi_{t+1}^k(\Delta_{t+1}, \Lambda^{k}_{t+1})(X_{t+1})=p(X_{t+1}~|~\Delta_{t+1}, \Lambda^{k}_{t+1})$ of member $k$ is the entire probability density function and not just its value at any particular realization of  $(\Delta_{t+1}, \Lambda_{t+1}^{k})$. This is because to compute $\Pi_{t+1}^k(\Delta_{t+1}, \Lambda^{k}_{t+1})(X_{t+1})$ for any particular $X_{t+1}$, we need the probability density functions $p(~\cdot ~|~ \Delta_{t}, \Lambda^{k}_t, U^{k}_t)$ and $p(~\cdot ~|~ \Delta_{t}, \Lambda^{k}_t)$. This implies that the information state takes values in the space of these probability densities on the measurable space $(\mathcal{X}_t, \mathscr{X}_t)$, which is an infinite-dimensional space.

\begin{lemma} \label{cor:theta}
	The information state of the team $\Pi_{t+1}(\Delta_{t+1}, \Lambda^{1:K}_{t+1})(X_{t+1})$ is a function of the information state $\Pi_{t+1}^k(\Delta_{t+1}, \Lambda^{k}_{t+1})(X_{t+1})$ of each team member $k\in\mathcal{K}$, $\Delta_{t+1}$, and $\Lambda^{1:K}_{t+1}$ for all $t=0,1,\ldots, T-1.$
	\begin{proof}
		 By applying Bayes' rule, for all $t=0,1,\ldots, T-1,$ we have
		\begin{align}\label{cor:eq:1}
	&\Pi_{t+1}(\Delta_{t+1}, \Lambda^{1:K}_{t+1})(X_{t+1})=p(X_{t+1}~|~\Delta_{t+1}, \Lambda^{1:K}_{t+1})\nonumber\\
	&= \frac{\splitfrac{p(\Lambda^{-k}_{t+1}~|~X_{t+1}, \Delta_{t+1}, \Lambda^{k}_{t+1})~ p(X_{t+1}~|~ \Delta_{t+1}, \Lambda^{k}_{t+1}) }{\cdot p(\Delta_{t+1}, \Lambda^{k}_{t+1})}}{p(\Delta_{t+1}, \Lambda^{1:K}_{t+1})}\nonumber\\	
		&= \frac{\splitfrac{p(\Lambda^{-k}_{t+1}~|~X_{t+1}, \Delta_{t+1}, \Lambda^{k}_{t+1})~ \Pi_{t+1}^k(\Delta_{t+1}, \Lambda^{k}_{t+1})(X_{t+1})}{\cdot p(\Delta_{t+1}, \Lambda^{k}_{t+1})}}{p(\Delta_{t+1}, \Lambda^{1:K}_{t+1})},	
\end{align}
		where $\Lambda^{-k}_{t}=(\Lambda_{t}^1,\ldots, \Lambda_{t}^{k-1}, \Lambda_{t}^{k+1}, \ldots,\Lambda_{t}^K)$.				 
		Thus, for each team member $k\in\mathcal{K}$, we can select an appropriate function $\vartheta_t^k$ such that
		\begin{align}
			&\Pi_{t+1}(\Delta_{t+1}, \Lambda^{1:K}_{t+1})(X_{t+1})\nonumber\\
			&=\vartheta^k_{t+1}\left(\Pi_{t+1}^k(\Delta_{t+1}, \Lambda^{k}_{t+1})(X_{t+1}),\Delta_{t+1}, \Lambda^{1:K}_{t+1}\right).\label{cor:eq:3}
		\end{align}
	
	\end{proof}
\end{lemma}

\begin{corollary}\label{cor:pi}
	For each team member $k\in\mathcal{K}$, the function $\vartheta_{t+1}^k$ is increasing with $\Pi_{t+1}^k(\Delta_{t+1}, \Lambda^{k}_{t+1})$.
\end{corollary}
\begin{proof}
	The function $\vartheta_{t+1}^k$ is continuous and differentiable in $[0, 1]$, while the variation of $\vartheta_{t+1}^k$ with respect to $\Pi_{t+1}^k(\Delta_{t+1}, \Lambda^{k}_{t+1})$ is positive.
\end{proof}

\subsection{Optimal Control Strategy of the Team Members}

In view of Theorem \ref{theo:y_tk}, we show that the optimal separated control strategy $\textbf{g}^{k}=\{g_0^k,\ldots,g_{T-1}^k\}$, i.e., $U_t^k=g_t^k(\Pi^k_{t}(\Delta_{t}, \Lambda^{k}_{t}))$, derived by each team member $k\in\mathcal{K}$ yields the same solution as the one by the manager's optimal separated control strategy $\textbf{g}\in\mathcal{G}^s$ (Theorem \ref{theo:dp_team}), and thus, the team members do not need a centralized intervention. 
We obtain a classical dynamic programming decomposition for member $k\in\mathcal{K}$ over the space of their decisions. 

\begin{lemma} \label{lemma:homo_member}
	Let $\textbf{g}^{k}$ be a separated control strategy of team member $k\in\mathcal{K}$ and 	
	$\textbf{g}=\{\textbf{g}^{1},\ldots, \textbf{g}^{K}\}$ be the team's control strategy. We fix $\textbf{g}^{-k}=(\textbf{g}^1,\ldots,\textbf{g}^{k-1}, \textbf{g}^{k+1}, \ldots,\textbf{g}^K)$, and let $V_t^k\big(\Pi_{t}(\Delta_{t}, \Lambda^{k}_{t})\big)$  be functions defined recursively for all $\textbf{g}$ by
	\begin{align}
		&V_T^k\big(\Pi_{T}(\Delta_{T}, \Lambda^{k}_{T})\big)\coloneqq \mathbb{E}^{\textbf{g}}\Big[c_T(X_T)~|~\Pi_{T}^k=\pi_T^k \Big],
		\label{homo_member:1a}
	\end{align}
	\begin{align}			
		&V_t^k\big(\Pi_{t}(\Delta_{t}, \Lambda^{k}_{t})\big)\coloneqq \inf_{u^{k}_t\in \mathcal{U}_t^k}\mathbb{E}^{\textbf{g}}\Big[c_t(X_t,U^{k}_t, U^{-k}_t)+ V_{t+1}^k\big(\Pi_{t+1}\nonumber\\ 
		& (\Delta_{t+1},\Lambda^{k}_{t+1})\big)~|~\Pi_{t}^k=\pi_t^k, \Delta_{t}=\delta_{t},  \Lambda^{k}_{t}=\lambda^{k}_{t}, U^{k}_t=u^{k}_t)  \Big],
		\label{homo_member:1b}
	\end{align}
	where $c_T(X_T)$ is the cost function at $T$, and $\pi_t^k$, $\delta_{t},$ $\lambda^{k}_{t}$, $u^{k}_t$ are the realizations of $\Pi_{t}^k$, $\Delta_{t}$, $\Lambda^{k}_{t}$, and $U^{k}_t$, respectively.
	Then, $V_t^k\big(\Pi_{t}(\Delta_{t}, \Lambda^{k}_{t})\big)$ is positive homogeneous, i.e., for any  $\rho >0$, $V_t^k\big(\rho~ \Pi_{t}(\Delta_{t}, \Lambda^{k}_{t})\big)= \rho ~V_t^k\big(\Pi_{t}(\Delta_{t}, \Lambda^{k}_{t})\big)$. 
\end{lemma}
\begin{proof}
	The proof is similar to the proof of Lemma \ref{lemma:homo}.
\end{proof}

\begin{theorem} \label{theo:concave_member}
	Let $\textbf{g}^{k}$ be a separated control strategy of team member $k\in\mathcal{K}$ and 	
	$\textbf{g}=\{\textbf{g}^{1},\ldots, \textbf{g}^{K}\}$ be the team's control strategy. We fix $\textbf{g}^{-k}=(\textbf{g}^1,\ldots,\textbf{g}^{k-1}, \textbf{g}^{k+1}, \ldots,\textbf{g}^K)$, and let $V_t^k\big(\Pi_{t}(\Delta_{t}, \Lambda^{k}_{t})\big)$ be functions defined recursively for all $\textbf{g}$ by
	\begin{align}			
		&V_t^k\big(\Pi_{t}(\Delta_{t}, \Lambda^{k}_{t})\big)\coloneqq \inf_{u^{k}_t\in \mathcal{U}_t^k}\mathbb{E}^{\textbf{g}}\Big[c_t(X_t,U^{k}_t, U^{-k}_t)\nonumber\\ 
		&+ V_{t+1}^k\big(\Pi_{t+1}(\Delta_{t+1}, \Lambda^{k}_{t+1})\big)~|~\Pi_{t}^k=\pi_t^k, \Delta_{t}=\delta_{t},  \nonumber\\
		&\Lambda^{k}_{t}=\lambda^{k}_{t}, U^{k}_t=u^{k}_t)  \Big],
		\label{theo_homo_member:1a}
	\end{align}
	where $\pi_t^k$, $\delta_{t},$ $\lambda^{k}_{t}$, $u^{k}_t$ are the realizations of $\Pi_{t}^k$, $\Delta_{t}$, $\Lambda^{k}_{t}$, and $U^{k}_t$, respectively.
	Then,  $V^k_t\big(\Pi_{t}(\Delta_{t}, \Lambda^{k}_{t})\big)$ is concave with respect to $\Pi_{t}(\Delta_{t}, \Lambda^{k}_{t})$.
\end{theorem}
\begin{proof}
	The proof is similar to the proof of Theorem \ref{theo:concave}.
\end{proof}

\begin{theorem} \label{theo:dp_team}
	Let $V_t\big(\Pi_{t}(\Delta_{t}, \Lambda^{1:K}_{t})\big)$ be functions defined recursively by
	\begin{align}
		&V_T\big(\Pi_{T}(\Delta_{T}, \Lambda^{1:K}_{T})\big)\coloneqq \mathbb{E}^{\textbf{g}}\Big[c_T(X_T)~|~\Pi_{T}=\pi_T \Big],
		\label{theo:dp_team:1aa}
	\end{align}
	\begin{align}			
		&V_t\big(\Pi_{t}(\Delta_{t}, \Lambda^{1:K}_{t})\big)\coloneqq \inf_{u^{1:K}_t\in\prod_{k\in\mathcal{K}} \mathcal{U}_t^k }\mathbb{E}^{\textbf{g}}\Big[c_t(X_t,U^{1:K}_t)\nonumber\\ 
		&+ V_{t+1}\big( \Pi_{t+1}(\Delta_{t+1}, \Lambda^{1:K}_{t+1})\big)~|~\Pi_{t}=\pi_t ,U^{1:K}_t=u^{1:K}_t \Big], \label{theo:dp_team:1a}
	\end{align}
	where $\pi_t$, $u^{1:K}_t$ are the realizations of $\Pi_{t}$ and $U^{1:K}_t$, respectively, and let  $\textbf{g}^{*}\in\mathcal{G}^s$  be the manager's optimal separated control strategy which achieves the infimum in \eqref{theo:dp_team:1aa}-\eqref{theo:dp_team:1a} for all $t=0,1,\ldots, T-1$.  
	Let $\textbf{g}=(\textbf{g}^{1},\ldots,\textbf{g}^{k-1}, \textbf{g}^{k}, \textbf{g}^{k+1}, \ldots,\textbf{g}^{K})$ be the team's strategy, where $\textbf{g}^{k}=\{g_0^k,\ldots,g_{T-1}^k\}$ is a separated control strategy of member  $k\in\mathcal{K}$ such that $U_t^k=g_t^k(\Pi^k_{t}(\Delta_{t}, \Lambda^{k}_{t}))$. Let $V_t^k\big(\Pi^k_{t}(\Delta_{t}, \Lambda^{k}_{t})\big)$ be functions defined recursively by each team player $k\in\mathcal{K}$ after fixing $\textbf{g}^{-k}=(\textbf{g}^1,\ldots,\textbf{g}^{k-1}, \textbf{g}^{k+1}, \ldots,\textbf{g}^K),$ by
	\begin{align}
		&V_T^k\big(\Pi^k_{T}(\Delta_{T}, \Lambda^{k}_{T})\big)\coloneqq \mathbb{E}^{\textbf{g}^k}\Big[c_T(X_T)~|~\Pi_{T}^k=\pi_T^k \Big],
		\label{theo:dp_team:1bb}
	\end{align}
	and	
	\begin{align}			
		&V_t^k\big(\Pi^k_{t}(\Delta_{t}, \Lambda^{k}_{t})\big)\coloneqq \inf_{u^{k}_t\in \mathcal{U}_t^k}\mathbb{E}^{\textbf{g}^k}\Big[c_t(X_t,U^{k}_t, U^{-k}_t)\nonumber\\ 
		&+ V_{t+1}^k\big(\Pi^k_{t+1}(\Delta_{t+1}, \Lambda^{k}_{t+1})\big)~|~\Pi_{t}^k=\pi_t^k, \Delta_{t}=\delta_{t},  \nonumber\\
		&\Lambda^{k}_{t}=\lambda^{k}_{t}, U^{k}_t=u^{k}_t)  \Big], \label{theo:dp_team:1b}
	\end{align}
where $U^{-k}_t=\big(U_{t}^1, \ldots, U_{t}^{k-1},U_{t}^{k+1},\ldots,U_{t}^K\big)$, and $\pi_t^k$, $\delta_{t},$ $\lambda^{k}_{t}$, $u^{k}_t$ are the realizations of $\Pi_{t}^k$, $\Delta_{t}$, $\Lambda^{k}_{t}$, and $U^{k}_t$, respectively.  Then, the solution of the manager  in \eqref{theo:dp_team:1aa}-\eqref{theo:dp_team:1a}  is the same as the solution derived by each player $k$  in \eqref{theo:dp_team:1bb}-\eqref{theo:dp_team:1b} for all $t=0,1,\ldots, T-1$. 
\end{theorem}
\begin{proof}
	See Appendix \ref{app:5}.	
\end{proof}

\section{Example} \label{sec:5}

We present an example of a delayed sharing pattern team consisting of two members ($K=2)$. This example was used  by Varaiya and Walrand \cite{Varaiya:1978aa} to show that Witsenhausen's structural result asserted in his seminal paper \cite{Witsenhausen1971a} is suboptimal.

\subsection{Problem Formulation}
In this example, the team evolves for a time horizon $T=3$ while there is a delay $n=2$ on  information sharing between the two team members. 
The state $X_t=(X_t^1, X_t^2), ~t=1,2,3,$ of the team is two-dimensional, and the initial state (primitive random variable), $X_0=(X_0^1, X_0^2),$ of the team is a Gaussian random variable with zero mean, variance $1$, and covariance $0.5$. 

The team's state evolves as follows:
\begin{align}\label{eq:example1}
	X_0 &= (X_0^1, X_0^2),\\
	X_1 &= (X_1^1, X_1^2) = (X_0^1 + X_0^2, 0),\\
	X_2 &= (X_2^1, X_2^2) = (X_1^1 , U_2^2) = (X_0^1 + X_0^2, U_2^2),\\			
	X_3 &= (X_3^1, X_3^2) = (X_2^1- X_2^2- U_3^1, 0)\nonumber \\
	&= (X_0^1 + X_0^2- U_2^2-U_3^1, 0).			\label{eq:example1b}
\end{align}
The observation equations are
\begin{align}\label{eq:example2}
	Y_t^k = X_{t-1}^k, \quad k=1,2;~ t=1,2,3.			
\end{align}
Each team member's feasible sets of actions $\mathcal{U}^k_t$ are specified by
\begin{align}\label{eq:example3}
	\mathcal{U}^k_t =
	\begin{cases}
		\mathbb{R}, \quad  \text{if } {(k,t)} = {{(1,3)\text{ or } (2,2)}}, \\
		{0}, \quad  \text{otherwise}. 
	\end{cases}			
\end{align}
Thus a control strategy $\textbf{g}=\{g_t^k;~ k=1,2;~ t=1,2,3\}, \textbf{g}\in\mathcal{G},$ of the team consists only of the pair $\textbf{g}=\{g_3^1, g_2^2\}$ since $g_t^k\equiv 0$ for the remaining $(k,t)$. Given the modeling framework above, the information structure $\{(\Delta_t, \Lambda_t^k);$ $k=1,2;~ t=1,2,3\}$ of the team is
\begin{align}\label{eq:example4}
	\Delta_1 &=\emptyset, ~\Delta_2 =\emptyset,\\
	\Delta_3 &=\{Y_1^1, Y_1^2\} = \{X_0^1, X_0^2\}.\label{eq:example4b}
\end{align}
Note that since $g_1^1\equiv 0$ and $g_1^2\equiv 0$, the realizations of $U_1^1$ and $U_1^2$ are zero, and thus $\Delta_3$ includes only the observations in \eqref{eq:example4b}. The data $\Lambda_t^k, k=1,2,$ available to the team member $k$ for the feasible control laws are
\begin{align}\label{eq:example5}
	\Lambda_2^2 &=\{Y_1^2, Y_2^2\} =\{X_0^2, X_1^2\}=\{X_0^2\},\\
	\Lambda_3^1 &=\{Y_2^1, Y_3^1\} =\{X_0^1+ X_0^2, X_0^1+X_0^2\}=\{X_0^1+X_0^2\}\label{eq:example5b}.
\end{align}

The problem is to derive the optimal control strategy $\textbf{g}^*=\{g_3^{1*}, g_2^{2*}\}$ which is the solution of 
\begin{align}\label{eq:example6}
	J(\textbf{g})&=\min_{u_2^2\in\mathcal{U}_2^2, u_3^1\in\mathcal{U}_3^1}\frac{1}{2}\mathbb{E}^{\textbf{g}}\left[(X_3^1)^2 + (U_3^1)^2\right]\nonumber\\
	&=\min_{u_2^2\in\mathcal{U}_2^2, u_3^1\in\mathcal{U}_3^1}\frac{1}{2}\mathbb{E}^{\textbf{g}}\left[(X_0^1+X_0^2-U_2^2-U_3^1)^2 + (U_3^1)^2 \right].
\end{align}

\subsection{Optimal Solution}
The feasible set $\mathcal{G}$ of control strategies of the team consists of all $\textbf{g}=\big\{g_3^1(\Lambda_3^1, \Delta_3), g_2^2(\Lambda_2^2, \Delta_2)\big\}$, i.e.,
\begin{align}\label{eq:example7}
	g_2^2&\colon \Delta_2\times \Lambda_2^2 \to U_2^2, ~ \text{ or }~ g_2^2\colon  X_0^2 \to \mathbb{R},\\
	g_3^1&\colon \Delta_3\times \Lambda_3^1 \to U_3^1, ~ \text{ or }~ g_3^1\colon  \{X_0^1, X_0^2\} \to \mathbb{R}.
\end{align}
The problem \eqref{eq:example6} has a unique optimal solution (see \cite{Varaiya:1978aa})
\begin{align}\label{eq:example8}
	U_2^2 = \frac{1}{2}X_0^2,\quad U_3^1=\frac{1}{2}(X_0^1+X_0^2) -\frac{1}{4}X_0^2.
\end{align}

\subsection{Solution Given by Theorem \ref{theo:dp2}}
Varaiya and Walrand \cite{Varaiya:1978aa} adopted the notation used by Witsenhausen \cite{Witsenhausen1971a} to describe chronologically the evolution of the team in their example,  which proceeds as follows: The initial state $X_0$ of the team is generated at $t=0$. Then, at the next time step $t=1$, each team member $k=1,2$ observes $Y_1^k=X_0$ and makes a decision $U_1^k =g_1^k(\Delta_1, \Lambda_1^k)$. The transition of the team to the next state occurs at the same time $t=1$, i.e., $X_1=f_1(X_0, U_1^{1:K})$, and the process is repeated until $t=T.$
Given that the decision of each team member $k$ at time $t$ depends on  $\Delta_{t}$ and $\Lambda_{t}^k$, and that the state of the team evolves after the realization of $U_t^{1:2}$, to be consistent, from a notation point of view, with the state equation \eqref{eq:state} in Section \ref{sec:2}-B, the evolution of the state of the team needs to be revised as follows:

\begin{align}\label{eq:example1a} 
	X_0 &= (X_0^1, X_0^2),\\
	X_1 &= (X_1^1, X_1^2)= (X_0^1, X_0^2),\\
	X_2 &= (X_2^1, X_2^2) = (X_0^1 + X_0^2, 0),\\
	X_3 &= (X_3^1, X_3^2) = (X_0^1 + X_0^2, U_2^2),\\			
	X_4 &= (X_4^1, X_4^2) = (X_3^1- X_3^2- U_3^1, 0)\nonumber \\
	&= (X_0^1 + X_0^2- U_2^2-U_3^1, 0),		
\end{align}
where we essentially included a degenerate transition at $t=1$.
Given the modeling framework presented in Section \ref{sec:2}-B, the information structure $\{(\Delta_t, \Lambda_t^k);$ $k=1,2; t=1,2,3\}$ of the team is
\begin{align}\label{eq:example55}
	\Delta_1 &=\emptyset, ~\Delta_2 =\emptyset,\\
	\Delta_3 &=\{Y_0^1, Y_0^2, Y_1^1, Y_1^2\} = \{X_0^1, X_0^2, X_1^1, X_1^2\}=\{X_0^1, X_0^2\}.\label{eq:example55b}
\end{align}
Since $g_1^1\equiv 0$ and $g_1^2\equiv 0$, the realizations of $U_1^1$ and $U_1^2$ are zero, and thus $\Delta_3$ includes only the observations in \eqref{eq:example5b}. The data $\Lambda_t^k, k=1,2,$ available to the team member $k$ for the feasible control laws are
\begin{align}\label{eq:example66}
	\Lambda_2^2 &=\{Y_0^2, Y_1^2, Y_2^2\} =\{X_0^2, X_1^2, X_2^2\}=\{X_0^2\},\\
	\Lambda_3^1 &=\{Y_2^1, Y_3^1\} =\{X_0^1+ X_0^2, X_0^1+ X_0^2\}=\{X_0^1+X_0^2\}\label{eq:example66b}.
\end{align}

We solve problem \eqref{eq:example6} by using the structural results presented in Section \ref{sec:3} considering the control strategies $\textbf{g}^k=\{g_t^k;~ k = 1,2; ~t = 0,1,2,3\}$ for each team member $k$, where the control law is of the form $g_t^k\big(\Pi(\Delta_t, \Lambda_t^k)\big)=g_t^k\big(\mathbb{P}(X_t~|~\Delta_t, \Lambda_t^k)\big)$. 

For $t=3$, the manager  formulates the dynamic program
\begin{align}\label{eq:example9}
	&V_3(\Pi_3)\nonumber\\
	&=\min_{u_2^2\in\mathcal{U}_2^2,u_3^1\in\mathcal{U}_3^1}\frac{1}{2}\mathbb{E}^{\textbf{g}^1}\Big [(X_0^1+X_0^2-U_2^2-U_3^1)^2 \nonumber\\
	&+ (U_3^1)^2~|~\Pi_3\big(\Delta_3, \Lambda_3^1\big), U_3^1\Big ]\nonumber\\
	&=\min_{u_2^2\in\mathcal{U}_2^2,u_3^1\in\mathcal{U}_3^1}\frac{1}{2}\mathbb{E}^{\textbf{g}^1}\Big [(X_0^1+X_0^2-U_2^2-U_3^1)^2 \nonumber\\
	&+ (U_3^1)^2~|~\mathbb{P}(X_0^1 + X_0^2, U_2^2~|~X_0^1, X_0^2, X_0^1+X_0^2), U_3^1\Big ],
\end{align}
where, for any given realization of $((X_0^1 + X_0^2), U_2^2)$ in the information state $\Pi_3$, the manager  selects $U_3^1$ to achieve the lower bound in \eqref{eq:example9}.
Thus,
\begin{align}\label{eq:example12}
	U_3^1=\frac{1}{2}(X_0^1+X_0^2) -\frac{1}{2}U_2^2.
\end{align}
Substituting \eqref{eq:example12} into  \eqref{eq:example9} yields
\begin{align}\label{eq:example13}
	&V_3(\Pi_3)= \min_{u_2^2\in\mathcal{U}_2^2,u_3^1\in\mathcal{U}_3^1}\frac{1}{2}\mathbb{E}^{\textbf{g}^1}\Big [\frac{\big(X_0^1 + X_0^2- U_2^2\big)^2}{2} ~| \nonumber\\
	&\mathbb{P}(X_0^1 + X_0^2, U_2^2~|~X_0^1, X_0^2, X_0^1+X_0^2), U_3^1\Big ].
\end{align}

For $t=2$, the manager formulates the following dynamic program
\begin{align}\label{eq:example14}
	&V_2(\Pi_2) \nonumber\\
	&= \min_{u_2^2\in\mathcal{U}_2^2,u_3^1\in\mathcal{U}_3^1} \frac{1}{2}\mathbb{E}^{\textbf{g}^2}\Big [V_3(\Pi_3)~|~~\Pi_2\big(\Delta_2, \Lambda_2^2\big), U_2^2 \Big ]\nonumber\\
		&= \min_{u_2^2\in\mathcal{U}_2^2,u_3^1\in\mathcal{U}_3^1} \frac{1}{2}\mathbb{E}^{\textbf{g}^2}\Big [V_3(\Pi_3)~|~\mathbb{P}(X_0^1 + X_0^2~|~X_0^2), U_2^2 \Big ]\nonumber\\
	&= \min_{u_2^2\in\mathcal{U}_2^2,u_3^1\in\mathcal{U}_3^1}\frac{1}{2}\mathbb{E}^{\textbf{g}^2}\Big [\frac{\big(X_0^1 + X_0^2- U_2^2\big)^2}{2} ~|~\mathbb{P}(X_0^1 + X_0^2~|\nonumber\\
	&~X_0^2), U_2^2\Big ].
\end{align}
Since
\begin{align}\label{eq:example15}
	U_2^2 = g_2^2\big(\mathbb{P}(X_2~|~\Delta_2, \Lambda_2^2)\big) = g_2^2\big(\mathbb{P}(X_0^1 + X_0^2~|~X_0^2)\big),
\end{align}
the problem of the manager in \eqref{eq:example14} is to choose, for any given $X_0^2$, the estimate of $(X_0^1+X_0^2)$ that minimizes the mean squared error $\big(X_0^1 + X_0^2- U_2^2\big)^2$ in \eqref{eq:example14}. Given the Gaussian statistics, the optimal solution is
\begin{align}\label{eq:example16}
	U_2^2  =\frac{1}{2}X_0^2.
\end{align}
Substituting \eqref{eq:example16} into \eqref{eq:example12} yields
\begin{align}\label{eq:example17}
	U_3^1=\frac{1}{2}(X_0^1+X_0^2) -\frac{1}{4}X_0^2.
\end{align}

Therefore, the control laws of the form $g_t^k\big(\Pi(\Delta_t, \Lambda_t^k)\big)=g_t^k\big(\mathbb{P}(X_t~|~\Delta_t, \Lambda_t^k)\big)$ yield the unique optimal solution \eqref{eq:example8} of problem \eqref{eq:example6}.

\section{Concluding Remarks and Discussion}
In this paper, we provided structural results and a classical dynamic programming decomposition of sequential dynamic team decision problems. We first addressed the problem from the point of view of a manager who seeks to derive the optimal strategy of a team in a centralized process. Then, we addressed the problem from the point of view of each team member, and showed that their solutions is the same as the ones derived by the manager. 
The key contributions of the paper are (1) the structural results for the team from the point of view of a manager that yield an information state which does not depend of the control strategy of the team, and (2) the structural results for each team member that yield an information state which does not depend on their control strategy. These results allow us to formulate two dynamic programming decompositions: (a) one for the team where the  manager's optimization problem is over the space of the team's decisions, and (b) one for each team member where the optimization problem is over the space of the decision of each  member. Finally, we showed that the solution of each team member is the same as the one derived by the manager. Therefore, each team member can derive their optimal strategy, which is also optimal for the team, without the manager's intervention. One particular limitation of the proposed approach is that the solution of the dynamic programming decompositions might become computationally intensive since the cost-to-go functions are defined on an infinite dimensional space. This a typical challenge in problems of partial observed Markov decision processes. However, given the characterization of the value functions (Theorem \ref{theo:concave} and Theorem \ref{theo:concave_member}) computationally efficient algorithms can be found. 

A potential direction for future research should explore the intersection of learning and control for team decision problems with nonclassical information structures. For example, cyber-physical systems, in most instances, represent systems of systems with informationally decentralized structure. In such systems, however, there is typically a large volume of data with a dynamic nature which is added to the system gradually and not altogether in advance. Therefore, neither traditional supervised (or unsupervised) learning nor typical model-based control approaches can effectively facilitate feasible solutions with performance guarantees. These challenges could be circumvented at the intersection of learning and control \cite{Malikopoulos2022a}. Given that the control strategies presented here are separated, a similar separation could be established between learning and control, and thus, combine the online and offline advantages of both traditional supervised (or unsupervised) learning and typical model-based control approaches.

\section{Acknowledgments}
The author would like to thank Aditya Mahajan, Ashutosh Nayyar,  and Serdar Yüksel for several helpful discussions that led to improving the exposition in Section IV.


\appendices
	\section{Proof of Theorem \ref{theo:y_t}}\label{app:1}
	By applying Bayes' rule, we have
\begin{align}
	&p^{\textbf{g}}(X_{t+1}~|~\Delta_{t+1}, \Lambda^{1:K}_{t+1})\nonumber\\
	&=\frac{\splitfrac{p^{\textbf{g}}(Y^{1:K}_{t+1}~|~X_{t+1}, \Delta_{t+1}, \Lambda^{1:K}_t, U^{1:K}_t)}{\cdot p^{\textbf{g}}(X_{t+1}, \Delta_{t+1}, \Lambda^{1:K}_t, U^{1:K}_t)}}{p^{\textbf{g}}(\Delta_{t+1}, \Lambda^{1:K}_{t+1})}\nonumber
\end{align}
\begin{align}
	&= \frac{p(Y^{1:K}_{t+1}~|~X_{t+1}) ~p^{\textbf{g}}(X_{t+1}, 	\Delta_{t+1}, \Lambda^{1:K}_t, U^{1:K}_t)}{p^{\textbf{g}}(\Delta_{t+1}, \Lambda^{1:K}_{t+1})}\nonumber\\
	&= \frac{\splitfrac{p(Y^{1:K}_{t+1}~|~X_{t+1}) ~p^{\textbf{g}}(X_{t+1} ~|~ \Delta_{t+1}, \Lambda^{1:K}_t,  U^{1:K}_t )}{\cdot p^{\textbf{g}}(\Delta_{t+1}, \Lambda^{1:K}_t,  U^{1:K}_t )}}{p^{\textbf{g}}(\Delta_{t+1}, \Lambda^{1:K}_{t+1})},	\label{eq:theo1a}
\end{align}
where in the second equality we used Lemma \ref{lem:y_t}. 

Next, 
\begin{multline} 
	p^{\textbf{g}}(\Delta_{t+1}, \Lambda^{1:K}_{t+1})= p^{\textbf{g}}(\Delta_{t+1}, \Lambda^{1:K}_t, Y^{1:K}_{t+1}, U^{1:K}_t)\nonumber\\
	= \int_{\mathscr{X}_{t+1}} p^{\textbf{g}}(X_{t+1}, \Delta_{t+1}, \Lambda^{1:K}_t, Y^{1:K}_{t+1}, U^{1:K}_t )~dX_{t+1}\nonumber\\
	=\int_{\mathscr{X}_{t+1}} p^{\textbf{g}}(Y^{1:K}_{t+1}~|~X_{t+1}, \Delta_{t+1}, \Lambda^{1:K}_t, U^{1:K}_t)\nonumber\\
	\cdot p^{\textbf{g}}(X_{t+1}, \Delta_{t+1}, \Lambda^{1:K}_t,  U^{1:K}_t )~dX_{t+1}\nonumber\\
	= \int_{\mathscr{X}_{t+1}} p^{\textbf{g}}(Y^{1:K}_{t+1}~|~X_{t+1}, \Delta_{t+1}, \Lambda^{1:K}_t, U^{1:K}_t)\nonumber\\
	\cdot p^{\textbf{g}}(X_{t+1} ~|~ \Delta_{t+1}, \Lambda^{1:K}_t,  U^{1:K}_t ) \nonumber \\
	\cdot p^{\textbf{g}}(\Delta_{t+1}, \Lambda^{1:K}_t,  U^{1:K}_t )~dX_{t+1},\label{eq:theo1b}
\end{multline}	
where by Lemma \ref{lem:y_t}, the last equation becomes
\begin{gather}
	p^{\textbf{g}}(\Delta_{t+1}, \Lambda^{1:K}_{t+1})\nonumber\\
	=\int_{\mathscr{X}_{t+1}} p(Y^{1:K}_{t+1}~|~X_{t+1})~p^{\textbf{g}}(X_{t+1} ~|~ \Delta_{t+1}, \Lambda^{1:K}_t,  U^{1:K}_t )\nonumber\\
	\cdot p^{\textbf{g}}(\Delta_{t+1}, \Lambda^{1:K}_t,  U^{1:K}_t )~dX_{t+1}.\label{eq:theo1ca}
\end{gather}
Note that $p^{\textbf{g}}(X_{t+1} ~|~ \Delta_{t+1},$ $\Lambda^{1:K}_t,  U^{1:K}_t )$ $=p^{\textbf{g}}(X_{t+1} ~|~ \Delta_{t},$ $ \Lambda^{1:K}_t,  U^{1:K}_t )$ since $Y^{1:K}_{t-n+1}$ and $U^{1:K}_{t-n+1}$ are already included in $\Lambda^{1:K}_{t},$ hence we can write \eqref{eq:theo1ca} as
\begin{gather}
	p^{\textbf{g}}(\Delta_{t+1}, \Lambda^{1:K}_{t+1})\nonumber\\
	=\int_{\mathscr{X}_{t+1}} p(Y^{1:K}_{t+1}~|~X_{t+1})~p^{\textbf{g}}(X_{t+1} ~|~ \Delta_{t}, \Lambda^{1:K}_t,  U^{1:K}_t )\nonumber\\
	\cdot p^{\textbf{g}}(\Delta_{t+1}, \Lambda^{1:K}_t,  U^{1:K}_t )~dX_{t+1}.\label{eq:theo1c}
\end{gather}	  
Substituting \eqref{eq:theo1c} into \eqref{eq:theo1a}, we have
\begin{gather}
	p^{\textbf{g}}(X_{t+1}~|~\Delta_{t+1}, \Lambda^{1:K}_{t+1})\nonumber\\
	\tiny = \frac{p(Y^{1:K}_{t+1}~|~X_{t+1}) ~p^{\textbf{g}}(X_{t+1} ~|~ \Delta_{t}, \Lambda^{1:K}_t,  U^{1:K}_t )}{\splitfrac{ \int_{\mathscr{X}_{t+1}} p(Y^{1:K}_{t+1}~|~X_{t+1})~	p^{\textbf{g}}(X_{t+1} ~|~ \Delta_{t}, \Lambda^{1:K}_t,}  {U^{1:K}_t ) ~dX_{t+1}}}, \label{eq:theo1h}
\end{gather}
which we can write as
\begin{align}
	p^{\textbf{g}}(X_{t+1}~|~\Delta_{t+1}, \Lambda^{1:K}_{t+1})\nonumber\\
	= \phi_t\big[p^{\textbf{g}}(\cdot ~|~ \Delta_{t}, \Lambda^{1:K}_t&, U^{1:K}_t), Y^{1:K}_{t+1} \big](X_{t+1}),
	\label{eq:theo1f}
\end{align}
with the function $\phi_t$ chosen appropriately. 

Next,
\begin{align}
	p^{\textbf{g}}(X_{t+1} ~|~ \Delta_{t}, \Lambda^{1:K}_t,  U^{1:K}_t )&\nonumber\\
	= \int_{\mathscr{X}_{t}} p^{\textbf{g}}(X_{t+1} ~|~ X_{t},\Delta_{t}, \Lambda^{1:K}_t,  U^{1:K}_t )&\nonumber\\
	\cdot p^{\textbf{g}}(X_{t} ~|~ \Delta_{t}, \Lambda^{1:K}_t,  &U^{1:K}_t )~dX_{t}.
	\label{eq:theo1d}
\end{align}	
By Lemma \ref{lem:x_t1ut} and Remark \ref{cor:lemU}, \eqref{eq:theo1d} becomes
\begin{gather}
	p^{\textbf{g}}(X_{t+1} ~|~ \Delta_{t}, \Lambda^{1:K}_t,  U^{1:K}_t )\nonumber\\
	= \int_{\mathscr{X}_{t}} p(X_{t+1} ~|~ X_t, U^{1:K}_t )~
	p(X_{t} ~|~ \Delta_{t}, \Lambda^{1:K}_t)~dX_{t},
	\label{eq:theo1e}
\end{gather}		
which we can write as
\begin{gather}
	p^{\textbf{g}}(X_{t+1} ~|~ \Delta_{t}, \Lambda^{1:K}_t,  U^{1:K}_t )\nonumber\\
	= \psi_t\big[ p(\cdot ~|~ \Delta_{t}, \Lambda^{1:K}_t),  U^{1:K}_t \big]\big( X_{t+1}\big),
	\label{eq:theo1i}
\end{gather}  
with the function $\psi_t$ chosen appropriately. 

Substituting \eqref{eq:theo1i} into \eqref{eq:theo1f} yields
\begin{gather}
	p^{\textbf{g}}(X_{t+1}~|~\Delta_{t+1}, \Lambda^{1:K}_{t+1})\nonumber\\
	= \phi_t\Big[\psi_t\big[ p(\cdot ~|~ \Delta_{t}, \Lambda^{1:K}_t),  U^{1:K}_t \big],
	Y^{1:K}_{t+1} \Big](X_{t+1}). \label{eq:theo1j}
\end{gather}  	
Therefore  $p^{\textbf{g}}(X_{t+1}~|~\Delta_{t+1}, \Lambda^{1:K}_{t+1})$ does not depend on the control strategy $\textbf{g}$, so we can drop the superscript. Moreover, we can choose appropriate function $\theta_t$ such that
\begin{gather}
	p(X_{t+1}~|~\Delta_{t+1}, \Lambda^{1:K}_{t+1})= \Pi_{t+1}(\Delta_{t+1}, \Lambda^{1:K}_{t+1})(X_{t+1}) \nonumber\\
	= \theta_t\big[ \Pi_{t}(\Delta_{t}, \Lambda^{1:K}_{t})(X_{t}), 	Y^{1:K}_{t+1}, U^{1:K}_t \big]. \label{eq:theo1k}
\end{gather}  	


\section{Proof of Lemma \ref{lemma:homo}}\label{app:2}

Obviously, for $t=T$,
\begin{align}			
	&V_T\big(\rho~ \Pi_{T}(\Delta_{T}, \Lambda^{1:K}_{T})\big)\nonumber\\
	&=  \int_{\mathscr{X}_T} c_T(X_T) ~\rho ~\Pi_{T}(\Delta_{T}, \Lambda^{1:K}_{T})(X_T)~ dX_T\nonumber\\
	& =\rho ~V_T\big(\Pi_{T}(\Delta_{T}, \Lambda^{1:K}_{T})\big).
	\label{theo_concave:1c}	
\end{align}	
For $t=0,\ldots,T-1$, by assigning $\Pi_{t}=\rho~ \Pi_{t}$ [recall $p(X_{t} ~|~ \Delta_{t}, \Lambda^{1:K}_t)=\Pi_{t}(\Delta_{t}, \Lambda^{1:K}_{t})$], \eqref{homo:1b} becomes
\begin{align}			
	&V_t\big(\rho ~\Pi_{t}(\Delta_{t}, \Lambda^{1:K}_{t})\big)\nonumber\\
	&= \inf_{u^{1:K}_t\in\prod_{k\in\mathcal{K}} \mathcal{U}_t^k }\bigg[ \int_{\mathscr{X}_t} c_t(X_t,U^{1:K}_t) ~\rho~ \Pi_{t}(\Delta_{t}, \Lambda^{1:K}_{t})(X_t)~ dX_t\nonumber\\
	&+ \int_{\mathscr{Y}_{t+1}} \int_{\mathscr{X}_{t+1}} \int_{\mathscr{X}_{t}} V_{t+1}\big(\rho~ \Pi_{t+1}(\Delta_{t+1}, \Lambda^{1:K}_{t+1})\big) \nonumber\\
	& \cdot p(Y^{1:K}_{t+1}~|~X_{t+1})~ p(X_{t+1} ~|~ X_t, U^{1:K}_t )~\rho ~p(X_{t} ~|~ \Delta_{t}, \Lambda^{1:K}_t)\nonumber\\
	&dX_{t}~dX_{t+1}~dY^{1:K}_{t+1}\bigg],
	\label{theo_concave:1d}	
\end{align}		
where $\mathscr{Y}_{t+1}=\otimes_{k\in\mathcal{K}}\mathscr{Y}^k$.

Next, from \eqref{eq:xt1}, 
\begin{align}
	&\rho~\Pi_{t+1}(\Delta_{t+1}, \Lambda^{1:K}_{t+1}) \nonumber\\
	&= \frac{\splitfrac{ p(Y^{1:K}_{t+1}~|~X_{t+1}) ~\int_{\mathscr{X}_{t}} p(X_{t+1} ~|~ X_t, U^{1:K}_t )~\bcancel{\rho}~p(X_{t} ~|~ \Delta_{t},} {\Lambda^{1:K}_t)~dX_{t}}}{\splitfrac{ \int_{\mathscr{X}_{t+1}} p(Y^{1:K}_{t+1}~|~X_{t+1})~	\int_{\mathscr{X}_{t}} p(X_{t+1} ~|~ X_t, U^{1:K}_t )~
		}  { \cdot \bcancel{\rho}~p(X_{t} ~|~ \Delta_{t},\Lambda^{1:K}_t)~dX_{t}~dX_{t+1}}}, \nonumber\\
	&= \Pi_{t+1}(\Delta_{t+1}, \Lambda^{1:K}_{t+1}).\label{theo_concave:1e}
\end{align} 
Substituting \eqref{theo_concave:1e} into \eqref{theo_concave:1d}, we have
\begin{align}			
	&V_t\big(\rho~ \Pi_{t}(\Delta_{t}, \Lambda^{1:K}_{t})\big)\nonumber\\
	&= \inf_{u^{1:K}_t\in\prod_{k\in\mathcal{K}} \mathcal{U}_t^k }\bigg[ \int_{\mathscr{X}_t} c_t(X_t,U^{1:K}_t) ~\rho~ \Pi_{t}(\Delta_{t}, \Lambda^{1:K}_{t})(X_t)~ dX_t\nonumber\\
	&+ \int_{\mathscr{Y}_{t+1}} \int_{\mathscr{X}_{t+1}} \int_{\mathscr{X}_{t}} V_{t+1}\big( \Pi_{t+1}(\Delta_{t+1}, \Lambda^{1:K}_{t+1})\big) \nonumber\\
	& \cdot p(Y^{1:K}_{t+1}~|~X_{t+1})~ p(X_{t+1} ~|~ X_t, U^{1:K}_t )~\rho ~p(X_{t} ~|~ \Delta_{t}, \Lambda^{1:K}_t)\nonumber\\
	&dX_{t}~dX_{t+1}~dY^{1:K}_{t+1}\bigg]\nonumber\\
	&=\rho~V_t\big(\Pi_{t}(\Delta_{t}, \Lambda^{1:K}_{t})\big).
	\label{theo_concave:1f}	
\end{align}		


\section{Proof of Theorem \ref{theo:concave}}\label{app:3}

	We  start with \eqref{theo_homo:1a} 
\begin{align}			
	&V_t\big(\Pi_{t}(\Delta_{t}, \Lambda^{1:K}_{t})\big)\nonumber\\
	&= \inf_{u^{1:K}_t\in\prod_{k\in\mathcal{K}} \mathcal{U}_t^k }\bigg[ \int_{\mathscr{X}_t} c_t(X_t,U^{1:K}_t) ~\Pi_{t}(\Delta_{t}, \Lambda^{1:K}_{t})(X_t)~ dX_t\nonumber\\
	&+ \int_{\mathscr{Y}_{t+1}} \int_{\mathscr{X}_{t+1}} \int_{\mathscr{X}_{t}} V_{t+1}\big( \Pi_{t+1}(\Delta_{t+1}, \Lambda^{1:K}_{t+1})\big) \nonumber\\
	& \cdot p(Y^{1:K}_{t+1}~|~X_{t+1})~ p(X_{t+1} ~|~ X_t, U^{1:K}_t )~p(X_{t} ~|~ \Delta_{t}, \Lambda^{1:K}_t)\nonumber\\
	&dX_{t}~dX_{t+1}~dY^{1:K}_{t+1}\bigg],
	\label{theo_homo:1b}
\end{align}		
where $\mathscr{Y}_{t+1}=\otimes_{k\in\mathcal{K}}\mathscr{Y}^k$.

Choosing 
\begin{align}	
	&\rho =  \int_{\mathscr{X}_{t+1}} \int_{\mathscr{X}_{t}} p(Y^{1:K}_{t+1}~|~X_{t+1})~\nonumber\\
	&\cdot p(X_{t+1} ~|~ X_t, U^{1:K}_t )~p(X_{t} ~|~ \Delta_{t},\Lambda^{1:K}_t)~dX_{t}~dX_{t+1},\label{theo_concave:1rho}	
\end{align}
we can use the positive homogeneity of $V_t\big(\Pi_{t}(\Delta_{t}, \Lambda^{1:K}_{t})\big)$ (Lemma \ref{lemma:homo}) to write the second part of \eqref{theo_homo:1b} as follows
\begin{align}			
	&\int_{\mathscr{Y}_{t+1}} \int_{\mathscr{X}_{t+1}} \int_{\mathscr{X}_{t}} V_{t+1}\big( \Pi_{t+1}(\Delta_{t+1}, \Lambda^{1:K}_{t+1})\big) \nonumber\\
	& \cdot p(Y^{1:K}_{t+1}~|~X_{t+1})~ p(X_{t+1} ~|~ X_t, U^{1:K}_t )~p(X_{t} ~|~ \Delta_{t}, \Lambda^{1:K}_t)\nonumber\\
	&dX_{t}~dX_{t+1}~dY^{1:K}_{t+1}\nonumber\\
	&=\int_{\mathscr{Y}_{t+1}} V_{t+1}\big(\rho~ \Pi_{t+1}(\Delta_{t+1}, \Lambda^{1:K}_{t+1})\big)~dY^{1:K}_{t+1}\nonumber\\
	&=\int_{\mathscr{Y}_{t+1}} V_{t+1}\bigg(
	p(Y^{1:K}_{t+1}~|~X_{t+1}) ~\int_{\mathscr{X}_{t}} p(X_{t+1} ~|~ X_t, U^{1:K}_t )\nonumber\\
	&p(X_{t} ~|~ \Delta_{t},\Lambda^{1:K}_t)~dX_{t}
	\bigg)~dY^{1:K}_{t+1},
	\label{theo_concave:1h}	
\end{align}
where, in the last equality, we substituted \eqref{theo_concave:1rho} and \eqref{eq:xt1}.

Thus, we can write \eqref{theo_homo:1b} as
\begin{align}			
	&V_t\big(\Pi_{t}(\Delta_{t}, \Lambda^{1:K}_{t})\big)\nonumber\\
	&= \inf_{u^{1:K}_t\in\prod_{k\in\mathcal{K}} \mathcal{U}_t^k }\bigg[ \int_{\mathscr{X}_t} c_t(X_t,U^{1:K}_t) ~\Pi_{t}(\Delta_{t}, \Lambda^{1:K}_{t})(X_t)~ dX_t\nonumber\\
	&+\int_{\mathscr{Y}_{t+1}} V_{t+1}\bigg(
	p(Y^{1:K}_{t+1}~|~X_{t+1}) ~\int_{\mathscr{X}_{t}} p(X_{t+1} ~|~ X_t, U^{1:K}_t )\nonumber\\
	&p(X_{t} ~|~ \Delta_{t},\Lambda^{1:K}_t)~dX_{t}
	\bigg)~dY^{1:K}_{t+1}\bigg].
	\label{theo_concave:1hh}	
\end{align}		

The remainder of the proof follows by induction. 
Suppose that $V_{t+1}\big( \Pi_{t+1}(\Delta_{t+1}, \Lambda^{1:K}_{t+1})\big)$ is concave. Since
\begin{align}
	&V_{t+1}\bigg(p(Y^{1:K}_{t+1}~|~X_{t+1}) ~\int_{\mathscr{X}_{t}} p(X_{t+1} ~|~ X_t, U^{1:K}_t )\nonumber\\
	&p(X_{t} ~|~ \Delta_{t},\Lambda^{1:K}_t)~dX_{t}\bigg),\label{theo_concave:1i}	
\end{align}
is the composition of a concave function and increasing linear function, it follows that it is concave. However, concavity is preserved by integration (see \cite{boyd2004}, p. 79), hence 
\begin{align}			
	&\int_{\mathscr{Y}_{t+1}} V_{t+1}\bigg(
	p(Y^{1:K}_{t+1}~|~X_{t+1}) ~\int_{\mathscr{X}_{t}} p(X_{t+1} ~|~ X_t, U^{1:K}_t )\nonumber\\
	&p(X_{t} ~|~ \Delta_{t},\Lambda^{1:K}_t)~dX_{t}
	\bigg)~dY^{1:K}_{t+1}
	\label{theo_concave:1j}	
\end{align}
is concave.
Since the pointwise infimum of concave functions is concave, \eqref{theo_concave:1hh} is concave.


\section{Proof of Theorem \ref{theo:y_tk}}\label{app:4}
By applying Bayes' rule, we have
\begin{align}
	&p^{\textbf{g}}(X_{t+1}~|~\Delta_{t+1}, \Lambda^k_{t+1}) \nonumber\\
	&=\frac{\splitfrac{p^{\textbf{g}}(Y^k_{t+1}~|~X_{t+1}, \Delta_{t+1}, \Lambda^k_t, U^k_t)}{\cdot p^{\textbf{g}}(X_{t+1}, \Delta_{t+1}, \Lambda^k_t, U^k_t)}}{p^{\textbf{g}}(\Delta_{t+1}, \Lambda^k_{t+1})}\nonumber\\
	&= \frac{p(Y^k_{t+1}~|~X_{t+1}) ~p^{\textbf{g}}(X_{t+1}, 	\Delta_{t+1}, \Lambda^k_t, U^k_t)}{p^{\textbf{g}}(\Delta_{t+1}, \Lambda^k_{t+1})}\nonumber\\
	&= \frac{\splitfrac{p(Y^k_{t+1}~|~X_{t+1}) ~p^{\textbf{g}}(X_{t+1} ~|~ \Delta_{t+1}, \Lambda^k_t,  U^k_t )}{\cdot p^{\textbf{g}}(\Delta_{t+1}, \Lambda^k_t,  U^k_t )}}{p^{\textbf{g}}(\Delta_{t+1}, \Lambda^k_{t+1})}, \label{eq:theo_y_tk_a}
\end{align}
where in the second equality we used Lemma \ref{lem:y_tk}.

Next, 
\begin{align}
	&p^{\textbf{g}}(\Delta_{t+1}, \Lambda^k_{t+1})= p^{\textbf{g}}(\Delta_{t+1}, \Lambda^k_t, Y^k_{t+1}, U^k_t)\nonumber\\
	&= \int_{\mathscr{X}_{t+1}} p^{\textbf{g}}(X_{t+1}, \Delta_{t+1}, \Lambda^k_t, Y^k_{t+1}, U^k_t )~dX_{t+1}\nonumber\\
	&=\int_{\mathscr{X}_{t+1}} p^{\textbf{g}}(Y^k_{t+1}~|~X_{t+1}, \Delta_{t+1}, \Lambda^k_t, U^k_t)\nonumber\\
	&\cdot p^{\textbf{g}}(X_{t+1}, \Delta_{t+1}, \Lambda^k_t,  U^k_t )~dX_{t+1}\nonumber\\
	&= \int_{\mathscr{X}_{t+1}} p^{\textbf{g}}(Y^k_{t+1}~|~X_{t+1}, \Delta_{t+1}, \Lambda^k_t, U^k_t)\nonumber\\
	&\cdot p^{\textbf{g}}(X_{t+1} ~|~ \Delta_{t+1}, \Lambda^k_t,  U^k_t )~p^{\textbf{g}}(\Delta_{t+1}, \Lambda^k_t,  U^k_t )~dX_{t+1},\label{eq:theo_y_tk_b}
\end{align}	
where by Lemma \ref{lem:y_tk}, the last equation becomes
\begin{gather}
	p^{\textbf{g}}(\Delta_{t+1}, \Lambda^k_{t+1})\nonumber\\
	=\int_{\mathscr{X}_{t+1}} p(Y^k_{t+1}~|~X_{t+1})~p^{\textbf{g}}(X_{t+1} ~|~ \Delta_{t+1}, \Lambda^k_t,  U^k_t )\nonumber\\
	\cdot p^{\textbf{g}}(\Delta_{t+1}, \Lambda^k_t,  U^k_t )~dX_{t+1}.\label{eq:theo_y_tk_c}
\end{gather}	  
Substituting \eqref{eq:theo_y_tk_c} into \eqref{eq:theo_y_tk_a}, we have
\begin{gather}
	p^{\textbf{g}}(X_{t+1}~|~\Delta_{t+1}, \Lambda^k_{t+1})\nonumber\\
	= \frac{p(Y^k_{t+1}~|~X_{t+1}) ~p^{\textbf{g}}(X_{t+1} ~|~ \Delta_{t+1}, \Lambda^k_t,  U^k_t )}{\int_{\mathscr{X}_{t+1}} p(Y^k_{t+1}~|~X_{t+1})~ p^{\textbf{g}}(X_{t+1} ~|~ \Delta_{t+1}, \Lambda^k_t,  U^k_t )~dX_{t+1}}, \label{eq:theo_y_tk_h}
\end{gather}
which we can write as
\begin{align}
	&p^{\textbf{g}}(X_{t+1} ~|~ \Delta_{t+1}, \Lambda^k_t,  U^k_t )\nonumber\\
	&= \phi^k_t\big[p^{\textbf{g}}(X_{t+1} ~|~ \Delta_{t+1}, \Lambda^k_t,  U^k_t ),Y^k_{t+1} \big](X_{t+1}),
	\label{eq:theo_y_tk_f}
\end{align}
with the function $\phi^k_t$ chosen appropriately.

By Lemma \ref{lem:x_t1k}, $p^{\textbf{g}}(X_{t+1} ~|~ \Delta_{t+1}, \Lambda^k_t,  U^k_t )$ depends only on the control strategy $\textbf{g}^{-k}$, so we can write $p^{\textbf{g}^{-k}}(X_{t+1} ~|~ \Delta_{t+1}, \Lambda^k_t,  U^k_t )$.
Next, 
\begin{align}
	p^{\textbf{g}^{-k}}(X_{t+1} ~|~ \Delta_{t+1}, \Lambda^k_t,  U^k_t )&\nonumber\\
	= \int_{\mathscr{X}_{t}} p^{\textbf{g}^{-k}}(X_{t+1} ~|~ X_{t},\Delta_{t+1}, \Lambda^k_t,  U^k_t )&\nonumber\\
	\cdot p^{\textbf{g}^{-k}}(X_{t} ~|~ \Delta_{t+1}&, \Lambda^k_t,  U^k_t )~dX_{t},
	\label{eq:theo_y_tk_d}
\end{align}	
where the last term in \eqref{eq:theo_y_tk_d} can be written as
\begin{align}
	&p^{\textbf{g}^{-k}}(X_{t} ~|~ \Delta_{t+1}, \Lambda^k_t,  U^k_t )= p^{\textbf{g}^{-k}}(X_{t} ~|~ \Delta_{t}, Y^{1:K}_{t-n+1}, U^{1:K}_{t-n+1},\nonumber\\
	&\Lambda^k_t) = p^{\textbf{g}^{-k}}(X_{t} ~|~ \Delta_{t}, Y^{-k}_{t-n+1}, U^{-k}_{t-n+1},\Lambda^k_t)\nonumber\\	
	&= p^{\textbf{g}^{-k}}(X_{t} ~|~ \Delta_{t}, \Lambda^k_t)\cdot p^{\textbf{g}^{-k}}(Y^{-k}_{t-n+1}, U^{-k}_{t-n+1}~|~ \Delta_{t}, \Lambda^k_t),
	\label{eq:theo_y_tk_e}
\end{align}
where $Y^{-k}_{t-n+1}=(Y_{t-n+1}^1,\ldots, Y_{t-n+1}^{k-1}, Y_{t-n+1}^{k+1}, \ldots,Y_{t-n+1}^K)$, and $U^{-k}_{t-n+1}=(U_{t-n+1}^1,\ldots, U_{t-n+1}^{k-1}, U_{t-n+1}^{k+1}, \ldots, U_{t-n+1}^K)$. In the second equality, we dropped $Y^{k}_{t-n+1}, U^{k}_{t-n+1}$ from conditioning since they  both are included in $\Lambda^k_t$.
Substituting \eqref{eq:theo_y_tk_e} into \eqref{eq:theo_y_tk_d} yields 

\begin{align}
	&p^{\textbf{g}^{-k}}(X_{t+1} ~|~ \Delta_{t+1}, \Lambda^k_t,  U^k_t )= \int_{\mathscr{X}_{t}} p^{\textbf{g}^{-k}}(X_{t+1} ~|~ X_{t},\Delta_{t+1}, \nonumber\\
	&\Lambda^k_t,  U^k_t ) \cdot p^{\textbf{g}^{-k}}(X_{t} ~|~ \Delta_{t}, \Lambda^k_t)\cdot p^{\textbf{g}^{-k}}(Y^{-k}_{t-n+1}, U^{-k}_{t-n+1}~|~ \Delta_{t}, \Lambda^k_t)\nonumber\\
	&\cdot dX_{t} = \int_{\mathscr{X}_{t}} p^{\textbf{g}^{-k}}(X_{t+1} ~|~ X_{t}, \Delta_{t}, Y^{1:K}_{t-n+1}, U^{1:K}_{t-n+1}, \Lambda^k_t,  U^k_t )\nonumber\\
	&\cdot p^{\textbf{g}^{-k}}(X_{t} ~|~ \Delta_{t}, \Lambda^k_t)\cdot p^{\textbf{g}^{-k}}(Y^{-k}_{t-n+1}, U^{-k}_{t-n+1}~|~ \Delta_{t}, \Lambda^k_t)~dX_{t}.	
	\label{eq:theo_y_tk_i}
\end{align}	

Substituting \eqref{eq:theo_y_tk_i} into \eqref{eq:theo_y_tk_f} yields
\begin{align}
	&p^{\textbf{g}}(X_{t+1}~|~\Delta_{t+1}, \Lambda^k_{t+1})\nonumber\\
	&= \phi^k_t\Big[\psi^k_t\big[ p^{\textbf{g}^{-k}}(\cdot ~|~ \Delta_{t}, \Lambda^k_t),   Y^{1:K}_{t-n+1}, U^{1:K}_{t-n+1},  \Delta_{t}, \Lambda^k_t, U^k_t \big],\nonumber\\
	&Y^k_{t+1} \Big](X_{t+1}).
	\label{eq:theo_y_tk_g}
\end{align}		

Therefore, $p^{\textbf{g}}(X_{t+1}~|~\Delta_{t+1}, \Lambda^k_{t+1})$,
does not depend on the control strategy $\textbf{g}^k$ of the team member $k$, so we can adjust the subscript accordingly. 
Moreover, we can choose appropriate function $\theta^k_t$ such that
\begin{align}
	&p^{\textbf{g}^{-k}}(X_{t+1}~|~\Delta_{t+1}, \Lambda^k_{t+1})=\Pi^k_{t+1}(\Delta_{t+1}, \Lambda^k_{t+1})(X_{t+1}) \nonumber\\
	&= \theta^k_t\big[ \Pi^k_{t}(\Delta_{t}, \Lambda^k_{t})(X_{t}), Y^{1:K}_{t-n+1}, U^{1:K}_{t-n+1},  \Delta_{t}, \Lambda^k_t, Y^k_{t+1}, U^k_t  \big]\nonumber\\
	&=\theta^k_t\big[ \Pi^k_{t}(\Delta_{t}, \Lambda^k_{t})(X_{t}),  \Delta_{t+1}, \Lambda^{k}_{t+1}  \big]. \label{eq:theo_y_tk_j}
\end{align}  	


\section{Proof of Theorem \ref{theo:dp_team}}\label{app:5}	
Let $\textbf{g}^*=\{\textbf{g}^{*1},\ldots, \textbf{g}^{*k-1}, \textbf{g}^{*k},$ $\textbf{g}^{*k+1},\ldots, \textbf{g}^{*K}\}$
be the optimal separated control strategy of the manager which achieves the infimum in \eqref{theo:dp_team:1aa}-\eqref{theo:dp_team:1a}. Starting with \eqref{theo:dp_team:1a}, we have

\begin{align}	
	&V_t\big(\Pi_{t}(\Delta_{t}, \Lambda^{1:K}_{t})\big) = \inf_{u^{1:K}_t\in\prod_{k\in\mathcal{K}} \mathcal{U}_t^k }\mathbb{E}^{\textbf{g}^*}\Big[c_t(X_t,U^{1:K}_t)\nonumber\\ 
	&+ V_{t+1}\big(\Pi_{t+1}(\Delta_{t+1}, \Lambda^{1:K}_{t+1}) \big)~|~\Pi_{t}=\pi_t , U^{1:K}_t=u^{1:K}_t \Big]\nonumber\\
	&=\inf_{u^{k}_t\in\mathcal{U}_t^k }	\inf_{u^{-k}_t\in\prod_{i\in\mathcal{K}\setminus\{k\}} \mathcal{U}_t^{i} }
	\mathbb{E}^{\textbf{g}^*}\Big[c_t(X_t,U^{k}_t,U^{-k}_t)\nonumber\\ 
	&+ V_{t+1}\big(\Pi_{t+1}(\Delta_{t+1}, \Lambda^{1:K}_{t+1}) \big)~|~\Pi_{t}=\pi_t , U^{1:K}_t=u^{1:K}_t \Big]\nonumber\\ 
	&=\inf_{u^{k}_t\in\mathcal{U}_t^k }	\inf_{u^{-k}_t\in\prod_{i\in\mathcal{K}\setminus\{k\}}  \mathcal{U}_t^{i} } \bigg[\int_{\mathscr{X}_{t}} c_t(X_t,U^{k}_t,U^{-k}_t)~\Pi_{t}(\Delta_{t}, ~\nonumber\\
	& \Lambda^{1:K}_{t})(X_t) ~dX_t\nonumber\\
	&+\mathbb{E}^{\textbf{g}^*}\Big[V_{t+1}\big(\Pi_{t+1}(\Delta_{t+1}, \Lambda^{1:K}_{t+1})\big)~|~\Pi_{t}=\pi_t ,	U^{1:K}_t=u^{1:K}_t\Big] \bigg].
	\label{theo:dp_team:1e}	
\end{align}

The function $V_t\big(\Pi_{t}(\Delta_{t}, \Lambda^{1:K}_{t})\big)$ is concave with respect to $\Pi_{t}(\Delta_{t}, \Lambda^{1:K}_{t})$ (Theorem \ref{theo:concave}) for all $t=0,1,\ldots, T$. Since  $\Pi_{t}(\Delta_{t}, \Lambda^{1:K}_{t})$ is increasing with $\Pi^k_{t}(\Delta_{t}, \Lambda^{k}_{t})$ for all $k$ (Corollary \ref{cor:pi}),  it follows that $V_t\big(\Pi_{t}(\Delta_{t}, \Lambda^{1:K}_{t})\big)$ is also increasing with respect to $\Pi^k_{t}(\Delta_{t}, \Lambda^{k}_{t})$ for all $k$. Thus, the manager can solve \eqref{theo:dp_team:1e}	for each member $k\in\mathcal{K}$ separately by fixing $u^{-k}_{t}=(u_{t}^1,\ldots, u_{t}^{k-1}, u_{t}^{k+1}, \ldots, u_{t}^K)$.
By substituting  \eqref{cor:eq:3} into \eqref{theo:dp_team:1e} and for any arbitrary $u^{-k}_{t}$, we have
\begin{align}	
	&V_t\big(\Pi_{t}(\Delta_{t}, \Lambda^{1:K}_{t})\big)=\inf_{u^{k}_t\in\mathcal{U}_t^k }	\bigg[\int_{\mathscr{X}_{t}} c_t(X_t,U^{k}_t,U^{-k}_t)\nonumber\\  &\cdot\vartheta^k_t\big(\Pi_{t}^k(\Delta_{t}, \Lambda^{k}_{t})(X_{t}), \Delta_{t}, \Lambda^{1:K}_{t}\big)~dX_t\nonumber\\		
	&+\mathbb{E}^{\textbf{g}^{*}}\Big[V_{t+1}\Big(\vartheta^k_{t+1}\big(\Pi_{t+1}^k(\Delta_{t+1}, \Lambda^{k}_{t+1})(X_{t+1}), \Delta_{t+1}, \Lambda^{1:K}_{t+1}\big)\Big)~\Big|\nonumber\\
	&~\Pi_{t}^k=\pi_t^k,
	\Delta_{t}=\delta_{t}, \Lambda^{1:K}_{t}=\lambda^{1:K}_{t}, U_t^k=u^{k}_t \Big]\bigg],
	\label{theo:dp_team:1fa}	
\end{align}		
where $\pi_t^k$, $\delta_{t}$, $\lambda^{1:K}_{t}$, and $U_t^k=u^{k}_t$ are the realizations of $\Pi_{t}^k$, $\Delta_{t}$, $\Lambda^{1:K}_{t}$, and $U_t^k$, respectively. Given that the function $\vartheta^k_t$ is increasing with  $\Pi^k_{t}(\Delta_{t}, \Lambda^{k}_{t})$ for all $k$ and $t$ (Corollary \ref{cor:pi}), from \eqref{theo:dp_team:1fa}, it follows that at $t=T-1$ and for any $u^{-k}_{T-1}$
\begin{align}	
	&\arginf_{u^{k}_{T-1}\in\mathcal{U}_{T-1}^k } \bigg[ \int_{\mathscr{X}_{{T-1}}} c_{T-1}(X_{T-1},U^{k}_{T-1},U^{-k}_{T-1})\nonumber\\
	&\cdot\vartheta^k_{T-1}\big(\Pi_{T-1}^k(\Delta_{T-1}, \Lambda^{k}_{T-1})(X_{T-1}),\Delta_{T-1},\Lambda^{1:K}_{T-1}\big)~dX_{T-1} \nonumber\\  
	&+\int_{\mathscr{X}_{{T}}} c_{T}(X_{T})\cdot\vartheta^k_{T}\big(\Pi_{T}^k(\Delta_{T}, \Lambda^{k}_{T})(X_{T}),\Delta_{T},\Lambda^{1:K}_{T}\big)~dX_{T} \bigg] \nonumber\\
	&=\arginf_{u^{k}_{T-1}\in\mathcal{U}_{T-1}^k } \bigg[ \int_{\mathscr{X}_{{T-1}}} c_{T-1}(X_{T-1},U^{k}_{T-1},U^{-k}_{T-1})\nonumber\\
	&\cdot\Pi_{T-1}^k(\Delta_{T-1}, \Lambda^{k}_{T-1})(X_{T-1})~dX_{T-1} \nonumber\\  
	&+\int_{\mathscr{X}_{{T}}} c_{T}(X_{T})\cdot\Pi_{T}^k(\Delta_{T}, \Lambda^{k}_{T})(X_{T})~dX_{T} \bigg], 
	\label{theo:dp_team:1g}	
\end{align}
or, alternatively, from \eqref{theo:dp_team:1aa}-\eqref{theo:dp_team:1a} and \eqref{theo:dp_team:1bb}-\eqref{theo:dp_team:1b},   \eqref{theo:dp_team:1g} can be written as
\begin{align}
	&\arginf_{u^{k}_{T-1}\in\mathcal{U}_{T-1}^k } \bigg[V_{T-1}\big(\Pi_{T-1}(\Delta_{T-1}, \Lambda^{1:K}_{T-1})\big)\bigg]	\nonumber\\
	&=\arginf_{u^{k}_{T-1}\in\mathcal{U}_{T-1}^k }\bigg[V_{T-1}^k\big(\Pi^k_{T-1}(\Delta_{T-1}, \Lambda^{k}_{T-1})\big)\bigg]
	\label{theo:dp_team:1ga}	
\end{align}		
Continuing backward in time, it follows that for all $t=0,1,\ldots, T$,  and for any $u^{-k}_{t}$, 
\begin{align}	
	&\arginf_{u^{k}_t\in\mathcal{U}_t^k }\bigg[ 	\int_{\mathscr{X}_{t}} c_t(X_t,U^{k}_t,U^{-k}_t)~\vartheta^k_t\big(\Pi_{t}^k(\Delta_{t}, \Lambda^{k}_{t})(X_{t}),\Delta_{t},\Lambda^{1:K}_{t}\big)~\nonumber\\  
	 	&\cdot dX_t+\mathbb{E}^{\textbf{g}^{*}}\Big[V_{t+1}\Big(\vartheta^k_{t+1}\big(\Pi_{t+1}^k(\Delta_{t+1}, \Lambda^{k}_{t+1})(X_{t+1}), \Delta_{t+1},\nonumber\\ &\Lambda^{1:K}_{t+1}\big)\Big)~\Big|~\Pi_{t}^k=\pi_t^k,
	\Delta_{t}=\delta_{t}, \Lambda^{1:K}_{t}=\lambda^{1:K}_{t}, U_t^k=u^{k}_t \Big]
	\bigg] \nonumber\\  
	&=\arginf_{u^{k}_t\in \mathcal{U}_t^k}\Bigg[\int_{\mathscr{X}_{t}} c_t(X_t,U^{k}_t,U^{-k}_t)~\Pi_{t}^k(\Delta_{t},\Lambda^{k}_{t})(X_t)~dX_t\nonumber\\
	&+	\mathbb{E}^{\textbf{g}^k}\bigg[V_{t+1}^k\big(\Pi^k_{t+1}(\Delta_{t+1}, \Lambda^{k}_{t+1})\big)~|~\Pi_{t}^k=\pi_t^k,\Delta_{t}=\delta_{t},  \nonumber\\
	&\Lambda^{k}_{t}=\lambda^{k}_{t}, U^{k}_t=u^{k}_t) \bigg]
	\Bigg].	
	\label{theo:dp_team:1gc}	
\end{align}

Therefore,  the solution of the manager in  \eqref{theo:dp_team:1aa}-\eqref{theo:dp_team:1a}  is the same as the solution derived by each member $k$  in  \eqref{theo:dp_team:1bb}-\eqref{theo:dp_team:1b} for all $t=0,1,\ldots, T-1$.

\balance

\bibliographystyle{IEEEtran}
\bibliography{TAC_Ref_structure, TAC_Ref_IDS, TAC_Ref_Andreas}

\begin{IEEEbiography}[{\includegraphics[width=1.1in,height=1.25in,clip,keepaspectratio]{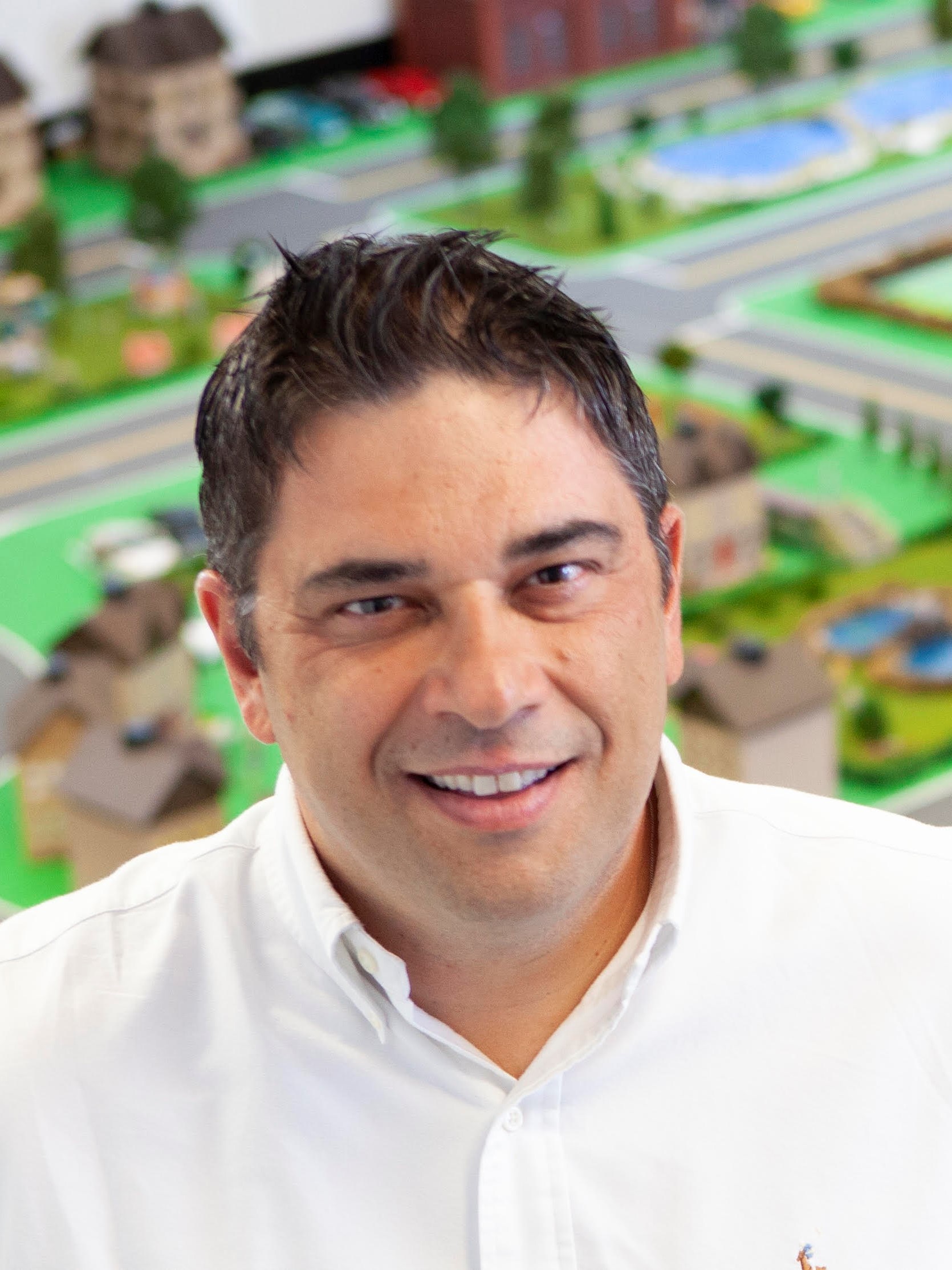}}]{Andreas A. Malikopoulos}
	(S'06--M'09--SM'17) received the Diploma in mechanical engineering from the National Technical University of Athens, Greece, in 2000. He received M.S. and Ph.D. degrees from the department of mechanical engineering at the University of Michigan, Ann Arbor, Michigan, USA, in 2004 and 2008, respectively. 
	He is the Terri Connor Kelly and John Kelly Career Development Associate Professor in the Department of Mechanical Engineering at the University of Delaware, the Director of the Information and Decision Science (IDS) Laboratory, and the Director of the Sociotechnical Systems Center. Prior to these appointments, he was the Deputy Director and the Lead of the Sustainable Mobility Theme of the Urban Dynamics Institute at Oak Ridge National Laboratory, and a Senior Researcher with General Motors Global Research \& Development. His research spans several fields, including analysis, optimization, and control of cyber-physical systems; decentralized systems; stochastic scheduling and resource allocation problems; and learning in complex systems. The emphasis is on applications related to smart cities, emerging mobility systems, and sociotechnical systems. He has been an Associate Editor of the IEEE Transactions on Intelligent Vehicles and IEEE Transactions on Intelligent Transportation Systems from 2017 through 2020. He is currently an Associate Editor of Automatica and IEEE Transactions on Automatic Control. He is a member of SIAM, AAAS, and a Fellow of the ASME.
\end{IEEEbiography}

\end{document}